\documentclass[11pt]{article}
\usepackage[latin1]{inputenc}
\usepackage{amsmath,amsthm,amssymb}
\usepackage{amsfonts}
\usepackage{amsmath,amsthm,amssymb,amscd}
\usepackage{latexsym}
\usepackage{color}
\usepackage{graphicx}
\usepackage{mathrsfs}
\usepackage{cite}
\usepackage[title]{appendix}

\usepackage{color,enumitem,graphicx}
\usepackage[colorlinks=true,urlcolor=black,
citecolor=black,linkcolor=black,linktocpage,pdfpagelabels,
bookmarksnumbered,bookmarksopen]{hyperref}

\makeatletter \@addtoreset{equation}{section} \makeatother

\newtheorem{theorem}{Theorem}[section]

\newtheorem{proposition}{Proposition}[section]
\newtheorem{lemma}{Lemma}[section]
\newtheorem{remark}{Remark}[section]

\allowdisplaybreaks

\begin{document}
\title{New type of solutions for the critical polyharmonic equation}

\author{Wenjing Chen\footnote{E-mail address:\, {\tt wjchen@swu.edu.cn} (W. Chen), {\tt zxwangmath@163.com} (Z. Wang).}\  \ and Zexi Wang\footnote{Corresponding author.}\\
\footnotesize  School of Mathematics and Statistics, Southwest University,
Chongqing, 400715, P.R. China}

\date{ }
\maketitle

\begin{abstract}
{ In this paper, we consider the following critical polyharmonic equation
\begin{align*}
(  -\Delta)^m u+V(|y'|,y'')u=u^{m^*-1},\quad u>0, \quad y=(y',y'')\in \mathbb{R}^3\times \mathbb{R}^{N-3},
  \end{align*}
where $m^*=\frac{2N}{N-2m}$, $N>4m+1$, $m\in \mathbb{N}^+$, and $V(|y'|,y'')$ is a bounded nonnegative function in $\mathbb{R}^+\times \mathbb{R}^{N-3}$. By using the reduction argument and local Poho\u{z}aev identities, we prove that if $r^{2m}V(r,y'')$ has a stable critical point $(r_0,y_0'')$ with $r_0>0$ and $V(r_0,y_0'')>0$, then the above problem has a new type of solutions, which concentrate at points lying on the top and the bottom circles of a cylinder. }

\vspace{.2cm}

\vspace{.2cm}

\end{abstract}

\section{Introduction}
In this paper, we consider the following nonlinear elliptic equation
\begin{align}\label{pro}
(  -\Delta)^m u+V(y)u=u^{m^*-1},\quad u>0,\quad u\in H^{m}(\mathbb{R}^N),
  \end{align}
where $m^*=\frac{2N}{N-2m}$, $N>4m+1$, $m\in \mathbb{N}^+$, $V(y)$ is a bounded nonnegative potential, and $H^{m}(\mathbb{R}^N)$ is the Hilbert space
with the scalar product
\begin{align*}
 \langle u,v\rangle=\left\{
  \begin{array}{ll}
  \displaystyle \int_{\mathbb{R}^N}\Big((\Delta^{\frac{m}{2}}u)(\Delta^{\frac{m}{2}}v)+uv \Big)dy,\quad &\text{if $m$ is even},
  \vspace{.3cm}\\ \displaystyle\int_{\mathbb{R}^N}\Big(\big(\nabla (\Delta^{\frac{m-1}{2}}u)\big)\cdot\big(\nabla(\Delta^{\frac{m-1}{2}}v)\big)+uv\Big) dy,\quad &\text{if $m$ is odd}.
    \end{array}
    \right.
  \end{align*}

In the case of $m=1$, problem \eqref{pro} reduces to
\begin{align}\label{mpro}
 -\Delta u+V(y)u=u^{2^*-1},\quad u>0,\quad u\in H^1(\mathbb{R}^N),
  \end{align}
where $2^*=\frac{2N}{N-2}$. Noting that if $V(y)\geq0$ and $V(y)\neq0$, the mountain pass value for \eqref{mpro} is not a critical value of the corresponding functional. Therefore, all the arguments based on the concentration compactness arguments \cite{Lions1,Lions2} can not be used to obtain an existence result of solutions for
\eqref{mpro}.
Benci and Cerami \cite{BC} first obtained the existence result for \eqref{mpro}, they proved that if $\|V\|_{\frac{N}{2}}$ 
is suitably small, then
\eqref{mpro} has a solution whose energy is in the interval $\big(\frac{1}{N}S^{\frac{N}{2}},\frac{2}{N}S^{\frac{N}{2}}\big)$, where $S$ is the best Sobolev constant in the embedding $D^{1,2}(\mathbb{R}^N)\hookrightarrow L^{2^*}(\mathbb{R}^N)$. In \cite{CWY}, Chen, Wei and Yan  proved that \eqref{mpro} has infinitely many solutions with arbitrarily large energy if $N \geq5$, $V(y)$ is radially symmetric, and $r^2V(r)$ has a local maximum point or a local minimum point $r_0>0$ with $V(r_0)>0$. In fact, this condition is necessary for the existence of solutions because by the Poho\u{z}aev identity
\begin{equation*}
  \int_{\mathbb{R}^N}\Big(V(|y|)+\frac{1}{2}|y|V'(|y|)\Big)u^2dy=0,
\end{equation*}
\eqref{mpro} admits no solution if $r^2V(r)$ is non-increasing or non-decreasing.
Later, Peng, Wang and Yan \cite{PWY} constructed infinitely many solutions  on a circle under a weaker symmetry condition for $V(y)$, where they only required that:

\vspace{.2cm}

\noindent$(A_1)$ $V(y)=V(|\tilde{y}'|,\tilde{y}'')$ is a bounded nonnegative function, where $y=(\tilde{y}',\tilde{y}'')\in \mathbb{R}^2\times \mathbb{R}^{N-2}$;

\vspace{.2cm}

\noindent$(A_2)$ $\tilde{r}^2V(\tilde{r},\tilde{y}'')$ has a stable critical point
$(\tilde{r}_0,\tilde{y}_0'')$ in the sense that $\tilde{r}^2V(\tilde{r},\tilde{y}'')$ has a critical point $(\tilde{r}_0,\tilde{y}_0'')$

\ \ \ \ satisfying $\tilde{r}_0>0$,
$V(\tilde{r}_0,\tilde{y}_0'')>0$, and $deg\big(\nabla (\tilde{r}^2V(\tilde{r},\tilde{y}'')),(\tilde{r}_0,\tilde{y}_0'')\big)\neq0$.

\vspace{.2cm}

\noindent Moreover, He, Wang and Wang \cite{HWW2} proved that the solutions obtained in \cite{PWY} are nondegenerate. For more related results about \eqref{mpro}, we refer the readers to \cite{GMYZ,HWW1,GL2,GLN2,VW,LT} and references therein.

Recently, Duan, Musso and Wei \cite{DMW} constructed a new type of solutions for the following equation
\begin{align*}
 -\Delta u=Q(y)u^{2^*-1},\quad u>0,\quad u\in H^1(\mathbb{R}^N),
  \end{align*}
where $Q(y)$ is radially symmetric, and the solutions concentrate at points lying on the top and the
bottom circles of a cylinder. More precisely, these solutions are different from \cite{PWY} and have the form
\begin{equation*}
  \sum\limits_{j=1}^kW_{\bar{x}_j,\lambda}+\sum\limits_{j=1}^kW_{\underline{x}_j,\lambda}+\varphi_k,
\end{equation*}
where $W_{x,\lambda}(y)=\big(\frac{\lambda}{1+\lambda^2|y-x|^2}\big)^{\frac{N-2}{2}}$, $\varphi_k$ is a remainder term,
\begin{align*}
   \left\{
  \begin{array}{ll}
  \bar{x}_j=\big(\bar{r}\sqrt{1-\bar{h}^2}\cos \frac{2(j-1)\pi}{k},\bar{r}\sqrt{1-\bar{h}^2}\sin \frac{2(j-1)\pi}{k},\bar{r}\bar{h},0\big),
  \quad &j=1,2,\cdots,k,\\
  \underline{x}_j=\big(\bar{r}\sqrt{1-\bar{h}^2}\cos \frac{2(j-1)\pi}{k},\bar{r}\sqrt{1-\bar{h}^2}\sin \frac{2(j-1)\pi}{k},-\bar{r}\bar{h},0\big),
  \quad &j=1,2,\cdots,k,
    \end{array}
    \right.
 \end{align*}
with $\bar{h}$ goes to zero, and $\bar{r}$ is close to some $r_0>0$.  Soon after, the above result was generalized to the fractional and polyharmonic cases, see \cite{DG} and \cite{GGH} for more details.
Inspired by the results of \cite{DMW} and \cite{PWY},  under a weaker symmetry condition on $V(y)$, by applying the reduction argument and local Poho\u{z}aev identities, Du et al. \cite{DHWW} obtained a new type of  solutions for \eqref{mpro}, which are different from those obtained in \cite{PWY}. Moreover, the fractional case was considered in \cite{DHW}.

If $m\geq 2$, \eqref{pro} becomes the equation involving polyharmonic operator. We would like to mention that the polyharmonic operator has found considerable interest in the literature due to its geometry roots, and the problems involving polyharmonic operator present new and changeling features compared with the elliptic operator (namely $m=1$).
To our best knowledge, few results are known for problem \eqref{pro}. In particular,
in \cite{GLN1}, Guo, Liu and Nie proved that problem \eqref{pro} has infinitely many solutions concentrated on a circle.

Motivated by \cite{DHWW}, \cite{DMW} and \cite{GLN1}, in this paper, under a weaker symmetry condition on $V(y)$, we intend to construct a new type of solutions for problem \eqref{pro}, which concentrate at points lying on the top and the bottom circles of a cylinder.

Before the statement of the main results, let us first introduce some notations. We assume that $V(y)$ satisfies the following assumptions:

\vspace{.2cm}

\noindent$(V_1)$ $V(y)=V(|y'|,y'')$ is a bounded nonnegative function, where $y=({y}',{y}'')\in \mathbb{R}^3\times \mathbb{R}^{N-3}$;

\begin{description}
\item [$(V_2)$] $r^{2m}V(r,y'')$ has a stable critical point $(r_0,y_0'')$ in the sense that $r^{2m}V(r,y'')$ has a critical point
 $({r}_0,{y}_0'')$
 satisfying ${r}_0>0$,
 $V(r_0,y_0'')>0$, and
$deg\big(\nabla (r^{2m}V(r,y'')),(r_0,y_0'')\big)\neq0$.
\end{description}

Recall that (see \cite{S})
\begin{equation*}
  U_{x,\lambda}(y)=P_{m,N}^{\frac{N-2m}{4m}}\Big(\frac{\lambda}{1+\lambda^2|y-x|^2}\Big)^{\frac{N-2m}{2}},\quad \lambda>0,\quad x\in \mathbb{R}^N,
\end{equation*}
is the only radial solution of the equation
\begin{align*}
(  -\Delta)^m u=u^{m^*-1},\quad u>0,\quad \text{ in $\mathbb{R}^N$},
  \end{align*}
where $P_{m,N}=\prod\limits_{h=-m}^{m-1}(N+2h)$ is a constant.

Define
\begin{align*}
  H_s=\Big\{u:&u\in H^{m}(\mathbb{R}^N),u(y_1,y_2,y_3,y'')=u(y_1,-y_2,-y_3,y''),\\
  &u(r \cos\theta,r \sin\theta, y_3,y'')=u\Big(r\cos \big(\theta+\frac{2j \pi}{k}\big),r\sin \big(\theta+\frac{2j \pi}{k}\big),y_3,y''\Big)\Big\},
\end{align*}
where $r=\sqrt{y_1^2+y_2^2}$ and $\theta=\arctan \frac{y_2}{y_1}$.

 Let
 \begin{align*}
   \left\{
  \begin{array}{ll}
  x_j^+=\big(\bar{r}\sqrt{1-\bar{h}^2}\cos \frac{2(j-1)\pi}{k},\bar{r}\sqrt{1-\bar{h}^2}\sin \frac{2(j-1)\pi}{k},\bar{r}\bar{h},\bar{y}''\big),
  \quad &j=1,2,\cdots,k,\\
  x_j^-=\big(\bar{r}\sqrt{1-\bar{h}^2}\cos \frac{2(j-1)\pi}{k},\bar{r}\sqrt{1-\bar{h}^2}\sin \frac{2(j-1)\pi}{k},-\bar{r}\bar{h},\bar{y}''\big),
  \quad &j=1,2,\cdots,k,
    \end{array}
    \right.
 \end{align*}
where $\bar{y}''$ is a vector in $\mathbb{R}^{N-3}$, $\bar{h}\in (0,1)$ and $(\bar{r},\bar{y}'')$ is close to $(r_0,y_0'')$.

In this paper, we consider the following three cases of $\bar{h}$ in the process of constructing solutions:
\vspace{.2cm}

$\bullet$ {\bf Case 1.} $\bar{h}$ goes to 1;

\vspace{.2cm}

$\bullet$ {\bf Case 2.} $\bar{h}$ is separated from 0 and 1;

\vspace{.2cm}

$\bullet$ {\bf Case 3.} $\bar{h}$ goes to 0.

\vspace{.2cm}
The idea of constructing solutions is to use $U_{x_j^{\pm},\lambda}$ as an approximation solution. To deal with the slow decay of this function when $N$ is not big enough, we introduce the smooth cut-off function $\xi(y)=\xi(|y'|,y'')$ satisfying $\xi=1$ if $|(r,y'')-(r_0,y_0'')|\leq \delta$, $\xi=0$ if $|(r,y'')-(r_0,y_0'')|\geq 2\delta$, and $0\leq \xi \leq 1$, where $\delta>0$ is a small constant such that $r^{2m}V(r,y'')>0$ if $|(r,y'')-(r_0,y_0'')|\leq 10\delta$.

Denote
\begin{equation*}
  Z_{x_j^{\pm},\lambda}=\xi U_{x_j^{\pm},\lambda},\quad Z^*_{\bar{r},\bar{h},\bar{y}'',\lambda}=\sum\limits_{j=1}^kU_{x_j^{+},\lambda}+\sum\limits_{j=1}^kU_{x_j^{-},\lambda},\quad
  Z_{\bar{r},\bar{h},\bar{y}'',\lambda}=\sum\limits_{j=1}^k\xi U_{x_j^{+},\lambda}+\sum\limits_{j=1}^k\xi U_{x_j^{-},\lambda}.
\end{equation*}

As for the {\bf Case 1}, we assume that $\alpha=N-4m-\iota$, $\iota$ is a small constant, $k>0$ is a large integer, $\lambda\in \big[L_0k^{\frac{N-2m}{N-4m-\alpha}},L_1k^{\frac{N-2m}{N-4m-\alpha}}\big]$ for some constants $L_1>L_0>0$ and $(\bar{r},\bar{h},\bar{y}'')$ satisfies
\begin{equation}\label{case1}
  |(\bar{r},\bar{y}'')-(r_0,y_0'')|\leq\vartheta,\quad \sqrt{1-\bar{h}^2}=M_1\lambda^{-\frac{\alpha}{N-2m}}+o(\lambda^{-\frac{\alpha}{N-2m}}),
\end{equation}
where $\vartheta$ is a small constant, $M_1$ is a positive constant.

\begin{theorem}\label{th1}
Assume that $N>4m+1$, if $V(y)$ satisfies $(V_1)$ and $(V_2)$, then there exists an integer $k_0>0$, such that for any $k>k_0$, equation \eqref{pro} has a solution $u_k$ of the form
\begin{equation*}
  u_k=Z_{\bar{r}_k,\bar{h}_k,\bar{y}_k'',\lambda_k}+\phi_k,
\end{equation*}
where $\lambda_k\in \big[L_0k^{\frac{N-2m}{N-4m-\alpha}},L_1k^{\frac{N-2m}{N-4m-\alpha}}\big]$ and $\phi_k\in H_s$. Moreover, as $k\rightarrow\infty$, $|(\bar{r}_k,\bar{y}_k'')-(r_0,y_0'')|\rightarrow0$, $\sqrt{1-\bar{h}_k^2}=M_1\lambda_k^{-\frac{\alpha}{N-2m}}+o(\lambda_k^{-\frac{\alpha}{N-2m}})$, and
$\lambda_k^{-\frac{N-2m}{2}}\|\phi_k\|_\infty\rightarrow0$.
\end{theorem}

\begin{remark}
{\rm The condition $N>4m+1$ in Theorem \ref{th1} is used in Lemmas \ref{err} and \ref{ener2} to guarantee the existence of a small constant $\iota>0$. }
\end{remark}

For the {\bf Case 2} and {\bf Case 3}, we assume that $k>0$ is a large integer, $\lambda\in \big[L_0'k^{\frac{N-2m}{N-4m}},L_1'k^{\frac{N-2m}{N-4m}}\big]$ for some constants $L_1'>L_0'>0$ and $(\bar{r},\bar{h},\bar{y}'')$ satisfies
\begin{equation}\label{case2}
  |(\bar{r},\bar{y}'')-(r_0,y_0'')|\leq\vartheta,\quad \bar{h}=a+{M_2\lambda^{-\frac{N-4m}{N-2m}}}+o(\lambda^{-\frac{N-4m}{N-2m}}),
\end{equation}
where $a\in [0,1)$, $\vartheta$ is a small constant, $M_2$ is a positive constant.

\begin{theorem}\label{th2}
Assume that $N>4m+1$, if $V(y)$ satisfies $(V_1)$ and $(V_2)$, then there exists an integer $k_0>0$, such that for any $k>k_0$, equation \eqref{pro} has a solution $u_k$ of the form
\begin{equation*}
  u_k=Z_{\bar{r}_k,\bar{h}_k,\bar{y}_k'',\lambda_k}+\phi_k.
\end{equation*}
where $\lambda_k\in \big[L_0k^{\frac{N-2m}{N-4m}},L_1k^{\frac{N-2m}{N-4m}}\big]$ and $\phi_k\in H_s$. Moreover, as $k\rightarrow\infty$, $|(\bar{r}_k,\bar{y}_k'')-(r_0,y_0'')|\rightarrow0$, $\bar{h}_k=a+{M_2\lambda_k^{-\frac{N-4m}{N-2m}}}+o(\lambda_k^{-\frac{N-4m}{N-2m}})$, and
$\lambda_k^{-\frac{N-2m}{2}}\|\phi_k\|_\infty\rightarrow0$.
\end{theorem}

\begin{remark}
{\rm The condition $N>4m+1$ in Theorem \ref{th2} is needed in the energy expansion, we can see this in Lemma \ref{ener2}.}
\end{remark}

\begin{remark}
{\rm The solutions obtained in Theorems \ref{th1} and \ref{th2} are different from those obtained in \cite{GLN1}.  }
\end{remark}

In this paper, we will prove Theorems \ref{th1} and \ref{th2} by using the finite dimensional reduction argument, introduced in \cite{BC1,FW}.
In the development of this method,
Wei and Yan \cite{WY1} first found that the number of the bubbles of solutions can be used to as the parameter of construction, and the idea of \cite{WY1} has been exploited and
further developed in the search for infinitely many solutions for equations with critical or asymptotic critical growth, see e.g. \cite{GN,LR,GL,DM,LWX,WY2,MWY,WY3,CYZ,GL1,PWW,WW}.

The paper is organized as follows. In Section \ref{two}, we will carry out the reduction procedure. Then we will study the reduced problem and prove Theorem \ref{th1} in Section \ref{three}. Section \ref{four} is devoted to the proof of Theorem \ref{th2}. In Appendix \ref{AppA}, we put some basic estimates. And we give the energy expansion for the approximation solutions in Appendix \ref{AppB}. Throughout the paper, $C$ denotes positive constant possibly different from line to line, $A=o(B)$ means $A/B\rightarrow 0$ and $A=O(B)$ means that $|A/B|\leq C$.

\section{Reduction argument}\label{two}
Let
\begin{equation*}
  \|u\|_*=\sup\limits_{y\in \mathbb{R}^N}\bigg(\sum\limits_{j=1}^k\Big(\frac{1}{(1+\lambda|y-x_j^+|)^{\frac{N-2m}{2}+\tau}}+
  \frac{1}{(1+\lambda|y-x_j^-|)^{\frac{N-2m}{2}+\tau}}
  \Big)\bigg)^{-1}\lambda^{-\frac{N-2m}{2}}|u(y)|,
\end{equation*}
and
\begin{equation*}
  \|f\|_{**}=\sup\limits_{y\in \mathbb{R}^N}\bigg(\sum\limits_{j=1}^k\Big(\frac{1}{(1+\lambda|y-x_j^+|)^{\frac{N+2m}{2}+\tau}}+
  \frac{1}{(1+\lambda|y-x_j^-|)^{\frac{N+2m}{2}+\tau}}
  \Big)\bigg)^{-1}\lambda^{-\frac{N+2m}{2}}|f(y)|,
\end{equation*}
where $\tau=\frac{N-4m-\alpha}{N-2m-\alpha}$. 
For $j=1,2,\cdots,k$, denote
\begin{equation*}
  Z_{j,2}^{\pm}=\frac{\partial Z_{x_j^\pm,\lambda}}{\partial \lambda},\quad Z_{j,3}^{\pm}=\frac{\partial Z_{x_j^\pm,\lambda}}{\partial \bar{r}},\quad Z_{j,l}^{\pm}=\frac{\partial Z_{x_j^\pm,\lambda}}{\partial \bar{y}_l''},\quad l=4,5,\cdots,N.
\end{equation*}

For later calculations, we divide $\mathbb{R}^N$ into $k$ parts, for $j=1,2,\cdots,k$, define
\begin{align*}
  \Omega_j:=\Big\{y:y=(y_1,y_2,y_3,y'')\in \mathbb{R}^3\times \mathbb{R}^{N-3},
  \Big\langle\frac{(y_1,y_2)}{|(y_1,y_2)|},\Big(\cos \frac{2(j-1)\pi}{k},\sin \frac{2(j-1)\pi}{k}\Big)\Big\rangle_{\mathbb{R}^2}\geq \cos \frac{\pi}{k}\Big\},
\end{align*}
where $\langle,\rangle_{\mathbb{R}^2}$ denotes the dot product in $\mathbb{R}^2$. For $\Omega_j$, we further divide it into two separate parts
\begin{equation*}
  \Omega_j^+:=\big\{y:y=(y_1,y_2,y_3,y'')\in \Omega_j,y_3\geq0\big\},
\end{equation*}
\begin{equation*}
  \Omega_j^-:=\big\{y:y=(y_1,y_2,y_3,y'')\in \Omega_j,y_3<0\big\}.
\end{equation*}
It's easy to verify that
\begin{equation*}
  \mathbb{R}^N=\bigcup\limits_{j=1}^k\Omega_j,\quad \Omega_j=\Omega_j^+\cup \Omega_j^-,
\end{equation*}
and
\begin{equation*}
  \Omega_j\cap \Omega_i=\emptyset,\quad \Omega_j^+\cap \Omega_j^-=\emptyset,\quad \text{if $i\neq j$}.
\end{equation*}

We also define the constrained space
\begin{equation*}
  \mathbb{H}:=\Big\{v:v\in H_s,\int_{\mathbb{R}^N}Z_{x_j^+,\lambda}^{m^*-2}Z_{j,l}^+vdy=0,
  \int_{\mathbb{R}^N}Z_{x_j^-,\lambda}^{m^*-2}Z_{j,l}^-vdy=0,\ \ j=1,2,\cdots,k,\ \ l=2,3,\cdots,N
  \Big\}.
\end{equation*}
Consider the following linearized problem
\begin{align}\label{lp}
 \left\{
  \begin{array}{ll}
 \ \ \ (-\Delta)^m \phi+V(r,y'')\phi-(m^*-1)Z_{\bar{r},\bar{h},\bar{y}'',\lambda}^{m^*-2}\phi\\
  =f+\sum\limits_{l=2}^Nc_l\sum\limits_{j=1}^k\Big
  (Z_{x_j^+,\lambda}^{m^*-2}Z_{j,l}^++Z_{x_j^-,\lambda}^{m^*-2}Z_{j,l}^-\Big),
  \quad  \mbox{in $\mathbb{R}^N$},\\
  \phi\in \mathbb{H},
    \end{array}
    \right.
\end{align}
for some real numbers $c_l$. 

In the sequel of this section, we assume that $(\bar{r},\bar{h},\bar{y}'')$ satisfies \eqref{case1}.

\begin{lemma}\label{xian}
Assume that $\phi_k$ solves \eqref{lp} for $f=f_k$. If $\|f_k\|_{**}$ goes to zero as $k$ goes to infinity, so does $\|\phi_k\|_*$.
\end{lemma}
\begin{proof}
Assume by contradiction that there exist $k\rightarrow\infty$, $\lambda_k\in \big[L_0k^{\frac{N-2m}{N-4m-\alpha}},L_1k^{\frac{N-2m}{N-4m-\alpha}}\big]$, $(\bar{r}_k,\bar{h}_k,\bar{y}_k'')$ satisfying \eqref{case1} and $\phi_k$ solving \eqref{lp} for $f=f_k$, $\lambda=\lambda_k$, $\bar{r}=\bar{r}_k$, $\bar{h}=\bar{h}_k$, $\bar{y}''=\bar{y}_k''$ with $\|f_k\|_{**}\rightarrow 0$ and $\|\phi_k\|_* \geq C>0$. Without loss of generality, we assume that $\|\phi_k\|_*=1$. For simplicity, we drop the subscript $k$.

From \eqref{lp}, we have
\begin{align*}
  |\phi(y)|\leq &C \int_{\mathbb{R}^N}\frac{1}{|y-z|^{N-2m}}Z_{\bar{r},\bar{h},\bar{y}'',\lambda}^{m^*-2}(z)|\phi(z)|dz+C \int_{\mathbb{R}^N}\frac{1}{|y-z|^{N-2m}}|f(z)|dz\\
  &+C \int_{\mathbb{R}^N}\frac{1}{|y-z|^{N-2m}}\Big|\sum\limits_{l=2}^Nc_l\sum\limits_{j=1}^k\Big
  (Z_{x_j^+,\lambda}^{m^*-2}Z_{j,l}^++Z_{x_j^-,\lambda}^{m^*-2}Z_{j,l}^-\Big)\Big|dz\\
  :=&I_1+I_2+I_3.
\end{align*}
By Lemma \ref{AppA3}, we deduce that
\begin{align*}
 I_1 \leq &C\|\phi\|_*\lambda^{\frac{N-2m}{2}}\int_{\mathbb{R}^N}\frac{Z_{\bar{r},\bar{h},\bar{y}'',\lambda}^{m^*-2}(z)}{|y-z|^{N-2m}}\sum\limits_{j=1}^k\Big(\frac{1}{(1+\lambda|z-x_j^+|)^{\frac{N-2m}{2}+\tau}}+
  \frac{1}{(1+\lambda|z-x_j^-|)^{\frac{N-2m}{2}+\tau}}\Big)
  dz\\
  \leq &C\|\phi\|_*\lambda^{\frac{N-2m}{2}}\sum\limits_{j=1}^k\Big(\frac{1}{(1+\lambda|y-x_j^+|)^{\frac{N-2m}{2}+\tau+\sigma}}+
  \frac{1}{(1+\lambda|y-x_j^-|)^{\frac{N-2m}{2}+\tau+\sigma}}\Big),
\end{align*}
where $\sigma>0$ is a small constant.

It follows from Lemma \ref{AppA2} that
\begin{align*}
   I_2\leq &C\|f\|_{**}\lambda^{\frac{N+2m}{2}}\int_{\mathbb{R}^N}\frac{1}{|y-z|^{N-2m}}\sum\limits_{j=1}^k\Big(\frac{1}{(1+\lambda|z-x_j^+|)^{\frac{N+2m}{2}+\tau}}+
  \frac{1}{(1+\lambda|z-x_j^-|)^{\frac{N+2m}{2}+\tau}}\Big)
  dz\\
  \leq &C\|f\|_{**}\lambda^{\frac{N-2m}{2}}\sum\limits_{j=1}^k\Big(\frac{1}{(1+\lambda|y-x_j^+|)^{\frac{N-2m}{2}+\tau}}+
  \frac{1}{(1+\lambda|y-x_j^-|)^{\frac{N-2m}{2}+\tau}}\Big).
\end{align*}

From Lemma \ref{AppA4}, we have
\begin{equation*}
 |Z_{j,2}^{\pm}|\leq C\lambda^{-\beta_1}Z_{x_j^\pm,\lambda},\quad |Z_{j,l}^{\pm}|\leq C\lambda Z_{x_j^\pm,\lambda},\quad l=3,4,\cdots,N,
\end{equation*}
where $\beta_1=\frac{\alpha}{N-2m}$. This with Lemma \ref{AppA2} yields
\begin{align*}
  I_3\leq & C\lambda^{\frac{N+2m}{2}+\eta_l}\sum\limits_{l=2}^N|c_l|\int_{\mathbb{R}^N}\frac{1}{|y-z|^{N-2m}}\sum\limits_{j=1}^k\Big(\frac{1}{(1+\lambda|z-x_j^+|)^{N+2m}}+
  \frac{1}{(1+\lambda|z-x_j^-|)^{N+2m}}\Big)
  dz\\
  \leq &C\lambda^{\frac{N-2m}{2}+\eta_l}\sum\limits_{l=2}^N|c_l|\sum\limits_{j=1}^k\Big(\frac{1}{(1+\lambda|y-x_j^+|)^{\frac{N-2m}{2}+\tau}}+
  \frac{1}{(1+\lambda|y-x_j^-|)^{\frac{N-2m}{2}+\tau}}\Big),
\end{align*}
where $\eta_2=-\beta_1$, $\eta_l=1$ for $l=3,4,\cdots,N$.

In the following, we estimate $c_l$, $l=2,3,\cdots,N$. Multiplying \eqref{lp} by $Z_{1,t}^+$ ($t=2,3,\cdots,N$), and integrating in $\mathbb{R}^N$, we have
\begin{align}\label{xiangu1}
  &\sum\limits_{l=2}^Nc_l\sum\limits_{j=1}^k\int_{\mathbb{R}^N}\Big
  (Z_{x_j^+,\lambda}^{m^*-2}Z_{j,l}^++Z_{x_j^-,\lambda}^{m^*-2}Z_{j,l}^-\Big)Z_{1,t}^+d y \nonumber\\
  =&\big\langle(-\Delta)^m \phi+V(r,y'')\phi-(m^*-1)Z_{\bar{r},\bar{h},\bar{y}'',\lambda}^{m^*-2}\phi,Z_{1,t}^+\big\rangle-\langle f, Z_{1,t}^+\rangle.
\end{align}
By the orthogonality, we get
\begin{equation}\label{xiangu2}
  \sum\limits_{j=1}^k\int_{\mathbb{R}^N}\Big
  (Z_{x_j^+,\lambda}^{m^*-2}Z_{j,l}^++Z_{x_j^-,\lambda}^{m^*-2}Z_{j,l}^-\Big)Z_{1,t}^+d y=(c_0+o(1))\delta_{l t}\lambda^{2\eta_l},
\end{equation}
for some constant $c_0>0$.

Using Lemmas \ref{AppA1} and \ref{AppA5}, we obtain
\begin{align}\label{toener1}
  &|\langle V(r,y'')\phi,Z_{1,t}^+\rangle| \nonumber\\
  \leq & C\|\phi\|_*\lambda^{{N-2m}+\eta_t}\int_{\mathbb{R}^N}\frac{1}{(1+\lambda|y-x_1^+|)^{N-2m}} \nonumber\\
  &\times\sum\limits_{j=1}^k\Big(\frac{1}{(1+\lambda|y-x_j^+|)^{\frac{N-2m}{2}+\tau}}+
  \frac{1}{(1+\lambda|y-x_j^-|)^{\frac{N-2m}{2}+\tau}}\Big)
  dy \nonumber\\
  \leq &C\|\phi\|_*\lambda^{{N-2m}+\eta_t}\int_{\mathbb{R}^N}\Big(\frac{1}{(1+\lambda|y-x_1^+|)^{\frac{3(N-2m)}{2}+\tau}}+\sum\limits_{j=2}^k\frac{1}{(1+\lambda|y-x_1^+|)^{N-2m}} \nonumber\\
  &\times \frac{1}{(1+\lambda|y-x_j^+|)^{\frac{N-2m}{2}+\tau}}+\sum\limits_{j=1}^k \frac{1}{(1+\lambda|y-x_1^+|)^{N-2m}}
  \frac{1}{(1+\lambda|y-x_j^-|)^{\frac{N-2m}{2}+\tau}}\Big)
  dy \nonumber\\
  \leq &C\|\phi\|_*\lambda^{{N-2m}+\eta_t}\bigg(\lambda^{-N}\max\Big\{\log\lambda,C,\lambda^{-\frac{N}{2}+3m-\tau}\Big\}+\lambda^{-N}\max\Big\{\log\lambda,C,\lambda^{-\frac{N}{2}+3m}\Big\}
  \sum\limits_{j=2}^k\frac{1}{(\lambda|x_j^+-x_1^+|)^{\tau}} \nonumber\\
  &+\lambda^{-N}\max\Big\{\log\lambda,C,\lambda^{-\frac{N}{2}+3m}\Big\}
  \sum\limits_{j=1}^k\frac{1}{(\lambda|x_j^--x_1^+|)^{\tau}}
  \bigg) \nonumber\\
   \leq &\frac{C\lambda^{\eta_t}\|\phi\|_*}{\lambda^{2m}}
   \Big(\max\Big\{\log\lambda,C,\lambda^{-\frac{N}{2}+3m-\tau}\Big\}+{{\max\Big\{\log\lambda,C,\lambda^{-\frac{N}{2}+3m}\Big\}}}
  \Big) \nonumber\\
  \leq & \frac{C\lambda^{\eta_t}\|\phi\|_*}{\lambda^{m+\varepsilon}},
\end{align}
where $\varepsilon>0$ is a small constant.

Similarly, we have
\begin{align*}
 | \langle f, Z_{1,t}^+\rangle|\leq &C\|f\|_{**}\lambda^{{N}+\eta_t}\int_{\mathbb{R}^N}\frac{1}{(1+\lambda|y-x_1^+|)^{N-2m}}\\
  &\times\sum\limits_{j=1}^k\Big(\frac{1}{(1+\lambda|y-x_j^+|)^{\frac{N+2m}{2}+\tau}}+
  \frac{1}{(1+\lambda|y-x_j^-|)^{\frac{N+2m}{2}+\tau}}\Big)
  dy\\
  \leq &C\lambda^{\eta_t}\|f\|_{**}.
\end{align*}

On the other hand, a direct computation gives
\begin{equation}\label{toener2}
  \big\langle(-\Delta)^m \phi-(m^*-1)Z_{\bar{r},\bar{h},\bar{y}'',\lambda}^{m^*-2}\phi,Z_{1,t}^+\big\rangle=O\Big(\frac{\lambda^{\eta_t}\|\phi\|_*}{\lambda^{m+\varepsilon}}\Big).
\end{equation}
Hence, we conclude that
\begin{equation*}
  \big\langle(-\Delta)^m \phi+V(r,y'')\phi-(m^*-1)Z_{\bar{r},\bar{h},\bar{y}'',\lambda}^{m^*-2}\phi,Z_{1,t}^+\big\rangle-\langle f, Z_{1,t}^+\rangle=O\Big(\lambda^{\eta_t}\big(\frac{\|\phi\|_*}{\lambda^{m+\varepsilon}}+\|f\|_{**}\big)\Big),
\end{equation*}
which together with \eqref{xiangu1} and \eqref{xiangu2} yields
\begin{equation*}
  c_l=\frac{1}{\lambda^{\eta_l}}\big(o(\|\phi\|_*)+O(\|f\|_{**})\big).
\end{equation*}
So
\begin{equation*}
  \|\phi\|_*\leq C\left(o(1)+\|f\|_{**}+\frac{\sum\limits_{j=1}^k\Big(\frac{1}{(1+\lambda|y-x_j^+|)^{\frac{N-2m}{2}+\tau+\sigma}}+
  \frac{1}{(1+\lambda|y-x_j^-|)^{\frac{N-2m}{2}+\tau+\sigma}}\Big)}{\sum\limits_{j=1}^k\Big(\frac{1}{(1+\lambda|y-x_j^+|)^{\frac{N-2m}{2}+\tau}}+
  \frac{1}{(1+\lambda|y-x_j^-|)^{\frac{N-2m}{2}+\tau}}\Big)}\right).
\end{equation*}
This with $\|\phi\|_*=1$ implies that there exists $R>0$ such that
\begin{equation}\label{xiandayu}
  \|\lambda^{-\frac{N-2m}{2}}\phi(y)\|_{L^\infty(B_{R/\lambda}(x_j^*))}\geq \widetilde{C}>0,
\end{equation}
for some $j$ with $x_j^*=x_j^+$ or $x_j^-$, where $\widetilde{C}$ is a positive constant. Furthermore, for this particular $j$, $\tilde{\phi}(y)=\lambda^{-\frac{N-2m}{2}}\phi(\lambda^{-1}y+x_j^*)$ converges uniformly on any compact set to a solution of the equation
\begin{equation}\label{xianlim}
  (-\Delta)^m u-(m^*-1)U_{0,\Lambda}^{m^*-2}u=0,\quad \text{in $\mathbb{R}^N$},
\end{equation}
for some $\Lambda\in [\Lambda_1,\Lambda_2]$ and $u$ is perpendicular to the kernel of  \eqref{xianlim}, according to the definition of $\mathbb{H}$. Hence, $u=0$, which contradicts \eqref{xiandayu}.
\end{proof}

With the help of Lemma \ref{xian}, the following result is a direct consequence of \cite[Proposition 4.1]{dFM}.
\begin{lemma}\label{dc}
There exists an integer $k_0>0$, such that for any $k\geq k_0$ and $f\in L^\infty(\mathbb{R}^N)$, problem \eqref{lp} has a unique solution $\phi=L_k(f)$. Moreover,
\begin{equation*}
  \|L_k(f)\|_*\leq C\|f\|_{**},\quad |c_l|\leq \frac{C}{\lambda^{\eta_l}}\|f\|_{**},
\end{equation*}
where $\eta_2=-\beta_1$, $\eta_l=1$ for $l=3,4,\cdots,N$.
\end{lemma}

Now, we consider the following problem
\begin{align}\label{pp}
 \left\{
  \begin{array}{ll}
  \ \ \ (-\Delta)^m (Z_{\bar{r},\bar{h},\bar{y}'',\lambda}+\phi)+V(r,y'')(Z_{\bar{r},\bar{h},\bar{y}'',\lambda}+\phi)\\
  =(Z_{\bar{r},\bar{h},\bar{y}'',\lambda}+\phi)_+^{m^*-1}+\sum\limits_{l=2}^Nc_l\sum\limits_{j=1}^k\Big
  (Z_{x_j^+,\lambda}^{m^*-2}Z_{j,l}^++Z_{x_j^-,\lambda}^{m^*-2}Z_{j,l}^-\Big),
  \quad  \mbox{in $\mathbb{R}^N$},\\
  \phi\in \mathbb{H}.
    \end{array}
    \right.
\end{align}

The main result of this section is
\begin{proposition}\label{fixed}
There exists an integer $k_0>0$, such that for any $k\geq k_0$, $\lambda\in \big[L_0k^{\frac{N-2m}{N-4m-\alpha}},L_1k^{\frac{N-2m}{N-4m-\alpha}}\big]$, $(\bar{r},\bar{h},\bar{y}'')$ satisfies \eqref{case1}, problem \eqref{pp} has a unique solution $\phi=\phi_{\bar{r},\bar{h},\bar{y}'',\lambda}$ satisfying
\begin{equation*}
  \|\phi\|_*\leq C\big(\frac{1}{\lambda}\big)^{\frac{2m+1-\beta_1}{2}+\varepsilon},\quad |c_l|\leq C\big(\frac{1}{\lambda}\big)^{\frac{2m+1-\beta_1}{2}+\eta_l+\varepsilon},
\end{equation*}
where $\varepsilon>0$ is a small constant.
\end{proposition}

Rewrite \eqref{pp} as
\begin{align}\label{repp}
 \left\{
  \begin{array}{ll}
 \ \ \ (-\Delta)^m \phi+V(r,y'')\phi-(m^*-1)Z_{\bar{r},\bar{h},\bar{y}'',\lambda}^{m^*-2}\phi\\
  =N(\phi)+E_k+\sum\limits_{l=2}^Nc_l\sum\limits_{j=1}^k\Big
  (Z_{x_j^+,\lambda}^{m^*-2}Z_{j,l}^++Z_{x_j^-,\lambda}^{m^*-2}Z_{j,l}^-\Big),
  \quad  \mbox{in $\mathbb{R}^N$},\\
  \phi\in \mathbb{H},
    \end{array}
    \right.
\end{align}
where
\begin{equation*}
  N(\phi)=(Z_{\bar{r},\bar{h},\bar{y}'',\lambda}+\phi)_+^{m^*-1}-Z_{\bar{r},\bar{h},\bar{y}'',\lambda}^{m^*-1}-(m^*-1)Z_{\bar{r},\bar{h},\bar{y}'',\lambda}^{m^*-2}\phi,
\end{equation*}
and
\begin{align*}
  E_k=&\underbrace{Z_{\bar{r},\bar{h},\bar{y}'',\lambda}^{m^*-1}-\sum\limits_{j=1}^k\Big(\xi U_{x_j^+,\lambda}^{m^*-1}+\xi U_{x_j^-,\lambda}^{m^*-1}\Big)}_{:=I_1}-\underbrace{V(r,y'')Z_{\bar{r},\bar{h},\bar{y}'',\lambda}}_{:=I_2}-\underbrace{\sum\limits_{l=0}^{m-1}\binom m l (-\Delta)^{m-l}\xi(-\Delta)^lZ^*_{\bar{r},\bar{h},\bar{y}'',\lambda}}_{:=I_3}\\
  &+\underbrace{2m \sum\limits_{l=0}^{m-1}\binom {m-1} l \nabla \big((-\Delta)^{m-l-1}\xi\big)\cdot \nabla \big((-\Delta)^lZ^*_{\bar{r},\bar{h},\bar{y}'',\lambda}\big)}_{:=I_4}\\
  &+\underbrace{\sum\limits_{l=1}^{m-1}a_l\sum\limits^N_{i_1,i_2,\cdots,i_{m-l}=1}\sum\limits_{s=0}^{l-1}\binom {l-1} s \nabla\Big(\frac{\partial ^{m-l}\big((-\Delta)^sZ^*_{\bar{r},\bar{h},\bar{y}'',\lambda}\big)}{\partial y_{i_{m-l}}\cdots \partial y_{i_1}}\Big)\cdot \nabla \Big(\frac{\partial ^{m-l}\big((-\Delta)^{l-s-1}\xi\big)}{\partial y_{i_{m-l}}\cdots \partial y_{i_1}}\Big)}_{:=I_5},
\end{align*}
where $a_l$, $l=1,2,\cdots,m-1$ are some constants.

In the following, we will make use of the contraction mapping theorem to prove that \eqref{repp} is uniquely solvable under the condition that $\|\phi\|_*$ is small enough, so we need to estimate $N(\phi)$ and $E_k$, respectively.

\begin{lemma}\label{non}
If $N>4m+1$, then
\begin{equation*}
  \|N(\phi)\|_{**}\leq C\|\phi\|_*^{\min\{2,m^*-1\}}.
\end{equation*}
\end{lemma}

\begin{proof}
It's easy to see
\begin{align*}
|N(\phi)|\leq  \left\{
  \begin{array}{ll}
  C|\phi|^{m^*-1},\quad &\text{if $N\geq 6m$},\\
  CZ_{\bar{r},\bar{h},\bar{y}'',\lambda}^{m^*-3}\phi^2+C|\phi|^{m^*-1},\quad &\text{if $4m+1<N\leq 6m-1$}.\\
    \end{array}
    \right.
  \end{align*}
Recall the definition of $\Omega_j^+$, by symmetry, we assume that $y\in \Omega_1^+$. Then it follows
\begin{equation}\label{da2}
  |y-x_j^+|\geq C|x_j^+-x_1^+|,\quad |y-x_j^-|\geq C|x_j^--x_1^+|,\quad j=1,2,\cdots,k.
\end{equation}

If $N\geq 6m$, by \eqref{da2} and the H\"{o}lder inequality
\begin{equation*}
  \sum\limits_{j=1}^ka_j b_j\leq \Big(\sum\limits_{j=1}^ka_j^p\Big)^{\frac{1}{p}} \Big(\sum\limits_{j=1}^kb_j^q\Big)^{\frac{1}{q}},\quad \frac{1}{p}+\frac{1}{q}=1, \quad a_j,b_j\geq0,
\end{equation*}
we obtain
\begin{align}\label{new1}
  |N(\phi)| \leq &C\|\phi\|_*^{{m^*-1}}\lambda^{\frac{N+2m}{2}}\bigg(\sum\limits_{j=1}^k\Big(\frac{1}{(1+\lambda|y-x_j^+|)^{\frac{N-2m}{2}+\tau}}+
  \frac{1}{(1+\lambda|y-x_j^-|)^{\frac{N-2m}{2}+\tau}}\Big)\bigg)^{m^*-1} \nonumber\\
  \leq &C \|\phi\|_*^{{m^*-1}}\lambda^{\frac{N+2m}{2}}\bigg(\sum\limits_{j=1}^k\Big(\frac{1}{(1+\lambda|y-x_j^+|)^{\frac{N+2m}{2}+\tau}}+
  \frac{1}{(1+\lambda|y-x_j^-|)^{\frac{N+2m}{2}+\tau}}\Big)\bigg) \nonumber\\
  &\times \bigg(\sum\limits_{j=1}^k\Big(\frac{1}{(1+\lambda|y-x_j^+|)^{\tau}}+
  \frac{1}{(1+\lambda|y-x_j^-|)^{\tau}}\Big)\bigg)^{m^*-2}\nonumber\\
  \leq &C \|\phi\|_*^{{m^*-1}}\lambda^{\frac{N+2m}{2}}\bigg(\sum\limits_{j=1}^k\Big(\frac{1}{(1+\lambda|y-x_j^+|)^{\frac{N+2m}{2}+\tau}}+
  \frac{1}{(1+\lambda|y-x_j^-|)^{\frac{N+2m}{2}+\tau}}\Big)\bigg)\nonumber\\
  &\times \bigg(1+\sum\limits_{j=2}^k\frac{1}{(\lambda|x_j^+-x_1^+|)^{\tau}}+\sum\limits_{j=1}^k\frac{1}{(\lambda|x_j^--x_1^+|)^{\tau}}\bigg)^{m^*-2}\nonumber\\
  \leq &C \|\phi\|_*^{{m^*-1}}\lambda^{\frac{N+2m}{2}}\sum\limits_{j=1}^k\Big(\frac{1}{(1+\lambda|y-x_j^+|)^{\frac{N+2m}{2}+\tau}}+
  \frac{1}{(1+\lambda|y-x_j^-|)^{\frac{N+2m}{2}+\tau}}\Big).
\end{align}
Hence, we obtain
$\|N(\phi)\|_{**}\leq C\|\phi\|_*^{m^*-1}$ for $N\geq 6m$.

If $4m+1<N\leq 6m-1$, we have
\begin{align*}
  |N(\phi)| \leq &C\|\phi\|_*^{{m^*-1}}\lambda^{\frac{N+2m}{2}}\bigg(\sum\limits_{j=1}^k\Big(\frac{1}{(1+\lambda|y-x_j^+|)^{\frac{N-2m}{2}+\tau}}+
  \frac{1}{(1+\lambda|y-x_j^-|)^{\frac{N-2m}{2}+\tau}}\Big)\bigg)^{m^*-1}\\
  &+C\|\phi\|_*^{{2}}\lambda^{\frac{N+2m}{2}}\bigg(\sum\limits_{j=1}^k\Big(\frac{1}{(1+\lambda|y-x_j^+|)^{N-2m}}+
  \frac{1}{(1+\lambda|y-x_j^-|)^{N-2m}}\Big)\bigg)^{m^*-3}\\
  &\times \bigg(\sum\limits_{j=1}^k\Big(\frac{1}{(1+\lambda|y-x_j^+|)^{\frac{N-2m}{2}+\tau}}+
  \frac{1}{(1+\lambda|y-x_j^-|)^{\frac{N-2m}{2}+\tau}}\Big)\bigg)^{2}
  \\
  \leq &C\|\phi\|_*^{{2}}\lambda^{\frac{N+2m}{2}}\bigg(\sum\limits_{j=1}^k\Big(\frac{1}{(1+\lambda|y-x_j^+|)^{\frac{N-2m}{2}+\tau}}+
  \frac{1}{(1+\lambda|y-x_j^-|)^{\frac{N-2m}{2}+\tau}}\Big)\bigg)^{m^*-1}\\
  \leq &C \|\phi\|_*^{{2}}\lambda^{\frac{N+2m}{2}}\sum\limits_{j=1}^k\Big(\frac{1}{(1+\lambda|y-x_j^+|)^{\frac{N+2m}{2}+\tau}}+
  \frac{1}{(1+\lambda|y-x_j^-|)^{\frac{N+2m}{2}+\tau}}\Big).
\end{align*}
Thus $\|N(\phi)\|_{**}\leq C\|\phi\|_*^{2}$ for  $4m+1<N\leq 6m-1$.
  \end{proof}

Next, we estimate $E_k$.

  \begin{lemma}\label{err}
If $N>4m+1$, then there exists a small constant $\varepsilon>0$ such that
\begin{equation*}
  \|E_k\|_{**}\leq C\big(\frac{1}{\lambda}\big)^{\frac{2m+1-\beta_1}{2}+\varepsilon}.
\end{equation*}
\end{lemma}

\begin{proof}
By symmetry, we assume that $y\in \Omega_1^+$. Then
\begin{equation}\label{da1}
   |y-x_j^-|\geq |y-x_j^+| \geq |y-x_1^+|,\quad j=1,2,\cdots,k.
\end{equation}
For $I_1$, we have
\begin{align*}
  |I_1|=&\Big|\Big(\sum\limits_{j=1}^k\big(\xi U_{x_j^+,\lambda}+\xi U_{x_j^-,\lambda}\big)\Big)^{m^*-1}-\sum\limits_{j=1}^k\Big(\xi U_{x_j^+,\lambda}^{m^*-1}+\xi U_{x_j^-,\lambda}^{m^*-1}\Big)\Big|\\
  \leq &C U_{x_1^+,\lambda}^{m^*-2}\Big( \sum\limits_{j=2}^kU_{x_j^+,\lambda}+\sum\limits_{j=1}^k U_{x_j^-,\lambda}\Big)+C\Big( \sum\limits_{j=2}^kU_{x_j^+,\lambda}+\sum\limits_{j=1}^k U_{x_j^-,\lambda}\Big)^{m^*-1}\\
  \leq &C \lambda^{\frac{N+2m}{2}}\frac{1}{(1+\lambda|y-x_1^+|)^{4m}}\Big(\sum\limits_{j=2}^k\frac{1}{(1+\lambda|y-x_j^+|)^{N-2m}}+
  \sum\limits_{j=1}^k\frac{1}{(1+\lambda|y-x_j^-|)^{N-2m}}\Big)\\
  &+C\lambda^{\frac{N+2m}{2}}\bigg(\sum\limits_{j=2}^k\frac{1}{(1+\lambda|y-x_j^+|)^{N-2m}}+
  \sum\limits_{j=1}^k\frac{1}{(1+\lambda|y-x_j^-|)^{N-2m}}\bigg)^{m^*-1}\\
  :=&I_{11}+I_{12}.
\end{align*}

For the term $I_{11}$, if $N-2m\geq \frac{N+2m}{2}-\tau$, by \eqref{da2}, \eqref{da1}, and Lemma \ref{AppA5}, we have
\begin{align*}
  I_{11}\leq &C \lambda^{\frac{N+2m}{2}}\frac{1}{(1+\lambda|y-x_1^+|)^{\frac{N+2m}{2}+\tau}}\Big(\sum\limits_{j=2}^k\frac{1}{(1+\lambda|y-x_j^+|)^{\frac{N+2m}{2}-\tau}}+
  \sum\limits_{j=1}^k\frac{1}{(1+\lambda|y-x_j^-|)^{\frac{N+2m}{2}-\tau}}\Big)\\
  \leq &C \lambda^{\frac{N+2m}{2}}\frac{1}{(1+\lambda|y-x_1^+|)^{\frac{N+2m}{2}+\tau}}\Big(\sum\limits_{j=2}^k\frac{1}{(\lambda|x_j^+-x_1^+|)^{\frac{N+2m}{2}-\tau}}+
  \sum\limits_{j=1}^k\frac{1}{(\lambda|x_j^--x_1^+|)^{\frac{N+2m}{2}-\tau}}\Big)\\
  \leq &C \lambda^{\frac{N+2m}{2}}\frac{1}{(1+\lambda|y-x_1^+|)^{\frac{N+2m}{2}+\tau}}\big(\frac{1}{\lambda}\big)^{\frac{2m}{N-2m}(\frac{N+2m}{2}-\tau)}\\
  \leq &C \lambda^{\frac{N+2m}{2}}\frac{1}{(1+\lambda|y-x_1^+|)^{\frac{N+2m}{2}+\tau}}\big(\frac{1}{\lambda}\big)^{\frac{2m+1-\beta_1}{2}+\varepsilon}.
\end{align*}
If $N-2m< \frac{N+2m}{2}-\tau$, which implies that $4m>\frac{N+2m}{2}+\tau$, then we obtain
\begin{align*}
  I_{11}\leq &C \lambda^{\frac{N+2m}{2}}\frac{1}{(1+\lambda|y-x_1^+|)^{\frac{N+2m}{2}+\tau}}\Big(\sum\limits_{j=2}^k\frac{1}{(1+\lambda|y-x_1^+|)^{N-2m}}+
  \sum\limits_{j=1}^k\frac{1}{(1+\lambda|y-x_1^-|)^{N-2m}}\Big)\\
  \leq &C \lambda^{\frac{N+2m}{2}}\frac{1}{(1+\lambda|y-x_1^+|)^{\frac{N+2m}{2}+\tau}}\big(\frac{1}{\lambda}\big)^{2m}\\
  \leq &C \lambda^{\frac{N+2m}{2}}\frac{1}{(1+\lambda|y-x_1^+|)^{\frac{N+2m}{2}+\tau}}\big(\frac{1}{\lambda}\big)^{\frac{2m+1-\beta_1}{2}+\varepsilon}.
\end{align*}
Hence,
\begin{equation}\label{err1}
  \|I_{11}\|_{**}\leq C\big(\frac{1}{\lambda}\big)^{\frac{2m+1-\beta_1}{2}+\varepsilon}.
\end{equation}

As for $I_{12}$, using \eqref{da2}, the H\"{o}lder inequality and Lemma \ref{AppA5}, we get
\begin{align*}
  I_{12}\leq  &C\lambda^{\frac{N+2m}{2}}\Big(\sum\limits_{j=2}^k\frac{1}{(1+\lambda|y-x_j^+|)^{\frac{N+2m}{2}+\tau}}\Big)\bigg(\sum\limits_{j=2}^k
  \frac{1}{(1+\lambda|y-x_j^+|)^{\frac{N+2m}{4m}(\frac{N-2m}{2}-\frac{N-2m}{N+2m}\tau)}}\bigg)^{m^*-2}\\
  &+C\lambda^{\frac{N+2m}{2}}\Big(\sum\limits_{j=1}^k\frac{1}{(1+\lambda|y-x_j^-|)^{\frac{N+2m}{2}+\tau}}\Big)\bigg(\sum\limits_{j=1}^k
  \frac{1}{(1+\lambda|y-x_j^-|)^{\frac{N+2m}{4m}(\frac{N-2m}{2}-\frac{N-2m}{N+2m}\tau)}}\bigg)^{m^*-2}\\
  \leq  &C\lambda^{\frac{N+2m}{2}}\Big(\sum\limits_{j=2}^k\frac{1}{(1+\lambda|y-x_j^+|)^{\frac{N+2m}{2}+\tau}}\Big)\bigg(\sum\limits_{j=2}^k
  \frac{1}{(\lambda|x_j^+-x_1^+|)^{\frac{N+2m}{4m}(\frac{N-2m}{2}-\frac{N-2m}{N+2m}\tau)}}\bigg)^{m^*-2}\\
  &+C\lambda^{\frac{N+2m}{2}}\Big(\sum\limits_{j=1}^k\frac{1}{(1+\lambda|y-x_j^-|)^{\frac{N+2m}{2}+\tau}}\Big)\bigg(\sum\limits_{j=1}^k
  \frac{1}{(\lambda|x_j^--x_1^+|)^{\frac{N+2m}{4m}(\frac{N-2m}{2}-\frac{N-2m}{N+2m}\tau)}}\bigg)^{m^*-2}\\
  \leq  &C\lambda^{\frac{N+2m}{2}}\Big(\sum\limits_{j=2}^k\frac{1}{(1+\lambda|y-x_j^+|)^{\frac{N+2m}{2}+\tau}}+\sum\limits_{j=1}^k\frac{1}{(1+\lambda|y-x_j^-|)^{
  \frac{N+2m}{2}+\tau}}\Big)\big(\frac{1}{\lambda}\big)^{\frac{2m}{N-2m}(\frac{N+2m}{2}-\tau)}\\
  \leq &C \lambda^{\frac{N+2m}{2}}\Big(\sum\limits_{j=2}^k\frac{1}{(1+\lambda|y-x_j^+|)^{\frac{N+2m}{2}+\tau}}+\sum\limits_{j=1}^k\frac{1}{(1+\lambda|y-x_j^-|)^{
  \frac{N+2m}{2}+\tau}}\Big)\big(\frac{1}{\lambda}\big)^{\frac{2m+1-\beta_1}{2}+\varepsilon}.
\end{align*}
Thus
\begin{equation}\label{err2}
  \|I_{12}\|_{**}\leq C\big(\frac{1}{\lambda}\big)^{\frac{2m+1-\beta_1}{2}+\varepsilon}.
\end{equation}

For $I_2$, we have
\begin{align*}
  I_2\leq &C \lambda^{\frac{N-2m}{2}} \sum\limits_{j=1}^k\Big(\frac{\xi}{(1+\lambda|y-x_j^+|)^{N-2m}}+
  \frac{\xi}{(1+\lambda|y-x_j^-|)^{N-2m}}\Big) \nonumber\\
  \leq &C \big(\frac{1}{\lambda}\big)^{\frac{2m+1-\beta_1}{2}+\varepsilon}\lambda^{\frac{N+2m}{2}} \sum\limits_{j=1}^k\Big(\frac{\xi}{\lambda^{\frac{2m-1+\beta_1}{2}-\varepsilon}(1+\lambda|y-x_j^+|)^{N-2m}}+
  \frac{\xi}{\lambda^{\frac{2m-1+\beta_1}{2}-\varepsilon}(1+\lambda|y-x_j^-|)^{N-2m}}\Big)\nonumber\\
  \leq &C \big(\frac{1}{\lambda}\big)^{\frac{2m+1-\beta_1}{2}+\varepsilon}\lambda^{\frac{N+2m}{2}} \sum\limits_{j=1}^k\Big(\frac{1}{(1+\lambda|y-x_j^+|)^{\frac{N+2m}{2}+\tau}}+
  \frac{1}{(1+\lambda|y-x_j^-|)^{\frac{N+2m}{2}+\tau}}\Big),
\end{align*}
where we used the fact that for any $|(r,y'')-(r_0,y_0'')|\leq 2\delta$,
\begin{equation}\label{fact1}
  \frac{1}{\lambda}\leq \frac{C}{1+\lambda|y-x_j^\pm|},
\end{equation}
and
$\frac{2m-1+\beta_1}{2}-\varepsilon>\frac{N+2m}{2}+\tau-(N-2m)$ if $\varepsilon>0$ small enough since $N>4m+1$ and $\iota$ is small.
Therefore, we have
\begin{equation}\label{err3}
  \|I_{2}\|_{**}\leq C\big(\frac{1}{\lambda}\big)^{\frac{2m+1-\beta_1}{2}+\varepsilon}.
\end{equation}

For $I_3$, we see from Lemma \ref{AppA7} that
\begin{align*}
  |I_3|\leq &C\lambda^{\frac{N-2m}{2}+2l}\sum\limits_{j=1}^k\sum \limits_{l=0}^{m-1}\frac{|(-\Delta)^{m-l}\xi|}{(1+\lambda|y-x_j^+|)^{N-2m+2l}}\\
  \leq &C\big(\frac{1}{\lambda}\big)^{\frac{2m+1-\beta_1}{2}+\varepsilon}\lambda^{\frac{N+2m}{2}}\sum\limits_{j=1}^k\sum \limits_{l=0}^{m-1}\frac{\lambda^{2l-\frac{2m-1+\beta_1}{2}+\varepsilon}|(-\Delta)^{m-l}\xi|}{(1+\lambda|y-x_j^+|)^{N-2m+2l}}.
\end{align*}
If $l=0,1,\cdots,[\frac{2m-1+\beta_1}{4}]$, then $2l-\frac{2m-1+\beta_1}{2}+\varepsilon\leq 0$, this with \eqref{fact1} yields
\begin{align*}
  \frac{\lambda^{2l-\frac{2m-1+\beta_1}{2}+\varepsilon}|(-\Delta)^{m-l}\xi|}{(1+\lambda|y-x_j^+|)^{N-2m+2l}}=& \frac{|(-\Delta)^{m-l}\xi|}{(1+\lambda|y-x_j^+|)^{N-2m+2l}\lambda^{-2l+\frac{2m-1+\beta_1}{2}-\varepsilon}}
  \leq \frac{C}{(1+\lambda|y-x_j^+|)^{\frac{N+2m}{2}+\tau}}.
\end{align*}
If $l=[\frac{2m-1+\beta_1}{4}]+1,\cdots, m-1$, then $2l-\frac{2m-1+\beta_1}{2}+\varepsilon> 0$, thus
\begin{align*}
  \frac{\lambda^{2l-\frac{2m-1+\beta_1}{2}+\varepsilon}|(-\Delta)^{m-l}\xi|}{(1+\lambda|y-x_j^+|)^{N-2m+2l}}
  \leq \frac{C}{(1+\lambda|y-x_j^+|)^{\frac{N+2m}{2}+\tau}},
\end{align*}
since for any $\delta\leq |(r,y'')-(r_0,y_0'')|\leq 2\delta$,
\begin{equation}\label{fact2}
   \frac{1}{1+\lambda|y-x_j^\pm|}\leq \frac{C}{\lambda},
\end{equation}
and
$N-2m-(\frac{N+2m}{2}+\tau)>-\frac{2m-1+\beta_1}{2}+\varepsilon$ if $\varepsilon>0$ small enough.
Hence,
\begin{equation}\label{err4}
  \|I_{3}\|_{**}\leq C\big(\frac{1}{\lambda}\big)^{\frac{2m+1-\beta_1}{2}+\varepsilon}.
\end{equation}

Similarly, for $I_4$ and $I_5$, Lemma \ref{AppA7} together with \eqref{fact1} and \eqref{fact2} leads to
\begin{align*}
  |I_4|\leq &C\lambda^{\frac{N-2m}{2}+2l+1}\sum\limits_{j=1}^k\sum \limits_{l=0}^{m-1}\frac{|\nabla\big((-\Delta)^{m-l-1}\xi\big)|}{(1+\lambda|y-x_j^+|)^{N-2m+2l+1}}\\
  \leq &C\big(\frac{1}{\lambda}\big)^{\frac{2m+1-\beta_1}{2}+\varepsilon}\lambda^{\frac{N+2m}{2}}\sum\limits_{j=1}^k\sum \limits_{l=0}^{m-1}\frac{\lambda^{2l+1-\frac{2m-1+\beta_1}{2}+\varepsilon}|\nabla\big((-\Delta)^{m-l}\xi\big)|}{(1+\lambda|y-x_j^+|)^{N-2m+2l+1}}\\
  \leq& C\big(\frac{1}{\lambda}\big)^{\frac{2m+1-\beta_1}{2}+\varepsilon}\lambda^{\frac{N+2m}{2}}\sum\limits_{j=1}^k\frac{1}{(1+\lambda|y-x_j^+|)^{\frac{N+2m}{2}+\tau}},
\end{align*}
and
\begin{align*}
  |I_5|\leq &C \lambda^{\frac{N-2m}{2}+m+l}\sum\limits_{j=1}^k\sum\limits_{l=1}^{m-1}a_l\sum\limits^N_{i_1,i_2,\cdots,i_{m-l}=1}\sum\limits_{s=0}^{l-1}\Big| \nabla \Big(\frac{\partial ^{m-l}\big((-\Delta)^{l-s-1}\xi\big)}{\partial y_{i_{m-l}}\cdots \partial y_{i_1}}\Big)\Big|\frac{1}{(1+\lambda|y-x_j^+|)^{N-m+l}}\\
  \leq& C\big(\frac{1}{\lambda}\big)^{\frac{2m+1-\beta_1}{2}+\varepsilon}\lambda^{\frac{N+2m}{2}}\sum\limits_{j=1}^k\sum\limits_{l=1}^{m-1}a_l\sum\limits^N_{i_1,i_2,\cdots,i_{m-l}=1}\sum\limits_{s=0}^{l-1}
  \Big| \nabla \Big(\frac{\partial ^{m-l}\big((-\Delta)^{l-s-1}\xi\big)}{\partial y_{i_{m-l}}\cdots \partial y_{i_1}}\Big)\Big|\frac{\lambda^{m+l-\frac{2m-1+\beta_1}{2}+\varepsilon}}{(1+\lambda|y-x_j^+|)^{N-m+l}}\\
  \leq& C\big(\frac{1}{\lambda}\big)^{\frac{2m+1-\beta_1}{2}+\varepsilon}\lambda^{\frac{N+2m}{2}}\sum\limits_{j=1}^k\frac{1}{(1+\lambda|y-x_j^+|)^{\frac{N+2m}{2}+\tau}}.
\end{align*}
As a result, we obtain
\begin{equation}\label{err5}
  \|I_{4}\|_{**}\leq C\big(\frac{1}{\lambda}\big)^{\frac{2m+1-\beta_1}{2}+\varepsilon},\quad \|I_{5}\|_{**}\leq C\big(\frac{1}{\lambda}\big)^{\frac{2m+1-\beta_1}{2}+\varepsilon}.
\end{equation}
Combining \eqref{err1}, \eqref{err2}, \eqref{err3}, \eqref{err4}, \eqref{err5}, we derive the conclusion.
\end{proof}

Now we are ready to prove Proposition \ref{fixed}.

\vspace{.3cm}
\noindent{\bf Proof of Proposition \ref{fixed}.} Denote
\begin{equation*}
  \mathbb{E}=\Big\{\phi:\phi\in C(\mathbb{R}^N)\cap \mathbb{H},\quad \|\phi\|_*\leq C\big(\frac{1}{\lambda}\big)^{\frac{2m+1-\beta_1}{2}}\Big\}.
\end{equation*}
By Lemma \ref{dc}, the existence and properties of the solution $\phi$ to \eqref{repp} is equivalent to find a fixed point for
\begin{equation}\label{fixp}
  \phi=T(\phi):=L_k(N(\phi)+E_k).
\end{equation}
Hence, it is sufficient to prove that $T$ is a contraction map from $\mathbb{E}$ to $\mathbb{E}$. In fact, for any $\phi\in \mathbb{E}$, by Lemmas \ref{dc}, \ref{non} and \ref{err}, we have
\begin{align*}
  \|T(\phi)\|_*\leq& C\|L_k(N(\phi))\|_{*}+\|L_k(E_k)\|_{*}\leq C\|N(\phi)\|_{**}+C\|E_k\|_{**}\\
  \leq& C\|\phi\|_*^{\min\{2,m^*-1\}}+C\big(\frac{1}{\lambda}\big)^{\frac{2m+1-\beta_1}{2}+\varepsilon}\leq C\big(\frac{1}{\lambda}\big)^{\frac{2m+1-\beta_1}{2}}.
\end{align*}
This shows that $T$ maps from $\mathbb{E}$ to $\mathbb{E}$.

On the other hand, for any $\phi_1,\phi_2\in \mathbb{E}$, we have
\begin{align*}
  \|T(\phi_1)-T(\phi_2)\|_*\leq C\|L_k(N(\phi_1))-L_k(N(\phi_2))\|_{*}\leq C\|N(\phi_1)-N(\phi_2)\|_{**}.
\end{align*}
If $N\geq 6m$, by \eqref{new1}, we have
\begin{align*}
  |N(\phi_1)-N(\phi_2)|\leq &C|N'(\phi_1+\kappa(\phi_2-\phi_1))||\phi_1-\phi_2|\leq C\big(|\phi_1|^{m^*-2}+|\phi_2|^{m^*-2}\big)|\phi_1-\phi_2|\\
   \leq& C\big(\|\phi_1\|_*^{m^*-2}+\|\phi_2\|_*^{m^*-2}\big)\|\phi_1-\phi_2\|_*\lambda^{\frac{N+2m}{2}}\\
   &\times\bigg(\sum\limits_{j=1}^k\Big(\frac{1}{(1+\lambda|y-x_j^+|)^{\frac{N-2m}{2}+\tau}}+
  \frac{1}{(1+\lambda|y-x_j^-|)^{\frac{N-2m}{2}+\tau}}\Big)\bigg)^{m^*-1}\\
  \leq &C \big(\|\phi_1\|_*^{m^*-2}+\|\phi_2\|_*^{m^*-2}\big)\|\phi_1-\phi_2\|_*\lambda^{\frac{N+2m}{2}}\\
  &\times\sum\limits_{j=1}^k\Big(\frac{1}{(1+\lambda|y-x_j^+|)^{\frac{N+2m}{2}+\tau}}+
  \frac{1}{(1+\lambda|y-x_j^-|)^{\frac{N+2m}{2}+\tau}}\Big),
\end{align*}
that is
\begin{align*}
  \|T(\phi_1)-T(\phi_2)\|_*\leq C \big(\|\phi_1\|_*^{m^*-2}+\|\phi_2\|_*^{m^*-2}\big)\|\phi_1-\phi_2\|_*<\frac{1}{2}\|\phi_1-\phi_2\|_*.
\end{align*}
Therefore, $T$ is a contraction map from $\mathbb{E}$ to $\mathbb{E}$. The case $4m+1<N\leq 6m-1$ can be discussed in a similar way.

By the contraction mapping theorem, there exists a unique $\phi=\phi_{\bar{r},\bar{h},\bar{y}'',\lambda}$ such that \eqref{fixp} holds. Moreover, by Lemmas  \ref{dc}, \ref{non} and \ref{err}, we deduce
\begin{align*}
  \|\phi\|_*\leq C\|L_k(N(\phi))\|_{*}+\|L_k(E_k)\|_{*}\leq C\|N(\phi)\|_{**}+C\|E_k\|_{**}\leq C\big(\frac{1}{\lambda}\big)^{\frac{2m+1-\beta_1}{2}+\varepsilon},
\end{align*}
and
\begin{equation*}
  |c_l|\leq  \frac{C}{\lambda^{\eta_l}}(\|N(\phi)\|_{**}+\|E_k\|_{**})\leq C\big(\frac{1}{\lambda}\big)^{\frac{2m+1-\beta_1}{2}+\eta_l+\varepsilon},
\end{equation*}
for $l=2,3,\cdots,N$. This completes the proof. \qed

\section{Proof of Theorem \ref{th1}}\label{three}
Recall that the functional corresponding to \eqref{pro} is
\begin{align*}
 I(u)=\left\{
  \begin{array}{ll}
  \displaystyle\frac{1}{2}\int_{\mathbb{R}^N}\Big(|\Delta^{\frac{m}{2}}u|^2+V(r,y'')u^2\Big)dy-\frac{1}{m^*}\int_{\mathbb{R}^N}(u)_+^{m^*}dy,\quad &\text{if $m$ is even},\vspace{.3cm}\\
  \displaystyle \frac{1}{2}\int_{\mathbb{R}^N}\Big(|\nabla (\Delta^{\frac{m-1}{2}}u)|^2+V(r,y'')u^2\Big)dy-\frac{1}{m^*}\int_{\mathbb{R}^N}(u)_+^{m^*}dy,\quad &\text{if $m$ is odd}.
    \end{array}
    \right.
  \end{align*}

Let
$\phi=\phi_{\bar{r},\bar{h},\bar{y}'',\lambda}$ be the function obtained in Proposition \ref{fixed} and $u_k=Z_{\bar{r},\bar{h},\bar{y}'',\lambda}+\phi$. In this section, we will choose suitable $(\bar{r},\bar{h},\bar{y}'',\lambda)$ such that $u_k$ is a solution of problem \eqref{pro}. For this purpose, we need the following result.
\begin{proposition}
Assume that $(\bar{r},\bar{h},\bar{y}'',\lambda)$ satisfies
\begin{equation}\label{con1}
  \int_{D_\varrho}\big((-\Delta)^mu_k+V(r,y'')u_k-(u_k)_+^{m^*-1}\big)\langle y, \nabla u_k\rangle dy=0,
\end{equation}
\begin{equation}\label{con2}
  \int_{D_\varrho}\big((-\Delta)^mu_k+V(r,y'')u_k-(u_k)_+^{m^*-1}\big)\frac{\partial u_k}{\partial y_i} dy=0,\quad i=4,5,\cdots,N,
\end{equation}
and
\begin{equation}\label{con3}
  \int_{\mathbb{R}^N}\big((-\Delta)^mu_k+V(r,y'')u_k-(u_k)_+^{m^*-1}\big)\frac{\partial Z_{\bar{r},\bar{h},\bar{y}'',\lambda}}{\partial \lambda} dy=0,
\end{equation}
where $u_k=Z_{\bar{r},\bar{h},\bar{y}'',\lambda}+\phi$ and $D_\varrho=\big\{(r,y''):|(r,y'')-(r_0,y_0'')|\leq \varrho\big\}$ with $\varrho\in (2\delta,5\delta)$, then $c_l=0$ for $l=2,3,\cdots,N$.
\end{proposition}
\begin{proof}
Since $Z_{\bar{r},\bar{h},\bar{y}'',\lambda}=0$ in $\mathbb{R}^N\backslash D_\varrho$, we see that if \eqref{con1}-\eqref{con3} hold, then
\begin{equation}\label{con'}
  \sum\limits_{l=2}^Nc_l\sum\limits_{j=1}^k\int_{\mathbb{R}^N}\Big
  (Z_{x_j^+,\lambda}^{m^*-2}Z_{j,l}^++Z_{x_j^-,\lambda}^{m^*-2}Z_{j,l}^-\Big)vdy=0,
\end{equation}
for $v=\langle y, \nabla u_k\rangle$, $v=\frac{\partial u_k}{\partial y_i}$ ($i=4,5,\cdots,N$), and $v=\frac{\partial Z_{\bar{r},\bar{h},\bar{y}'',\lambda}}{\partial \lambda}$.

By direct computations, we can prove that
\begin{equation}\label{pr1}
  \sum\limits_{j=1}^k\int_{\mathbb{R}^N}\Big
  (Z_{x_j^+,\lambda}^{m^*-2}Z_{j,3}^++Z_{x_j^-,\lambda}^{m^*-2}Z_{j,3}^-\Big)\langle y', \nabla_{y'} Z_{\bar{r},\bar{h},\bar{y}'',\lambda}\rangle dy=2k\lambda^2 (a_1+o(1)),
\end{equation}
\begin{equation}\label{pr2}
  \sum\limits_{j=1}^k\int_{\mathbb{R}^N}\Big
  (Z_{x_j^+,\lambda}^{m^*-2}Z_{j,i}^++Z_{x_j^-,\lambda}^{m^*-2}Z_{j,i}^-\Big)\frac{\partial Z_{\bar{r},\bar{h},\bar{y}'',\lambda}}{\partial y_i}dy=2k \lambda^2(a_2+o(1)),\quad i=4,5,\cdots,N,
\end{equation}
and
\begin{equation}\label{pr3}
  \sum\limits_{j=1}^k\int_{\mathbb{R}^N}\Big
  (Z_{x_j^+,\lambda}^{m^*-2}Z_{j,2}^++Z_{x_j^-,\lambda}^{m^*-2}Z_{j,2}^-\Big)\frac{\partial Z_{\bar{r},\bar{h},\bar{y}'',\lambda}}{\partial \lambda}dy=\frac{2k }{\lambda^{2\beta_1}} (a_3+o(1)),
\end{equation}
for some constants $a_1\neq0$, $a_2\neq0$, and $a_3>0$.

Integrating by parts, we get
\begin{equation*}
  \sum\limits_{l=2}^Nc_l\sum\limits_{j=1}^k\int_{\mathbb{R}^N}\Big
  (Z_{x_j^+,\lambda}^{m^*-2}Z_{j,l}^++Z_{x_j^-,\lambda}^{m^*-2}Z_{j,l}^-\Big)vdy=o(k\lambda^2)\sum\limits_{l=3}^N|c_l|+o(k\lambda^{1-\beta_1}|c_2|),
\end{equation*}
for $v=\langle y, \nabla \phi_{\bar{r},\bar{h},\bar{y}'',\lambda}\rangle$ and $v=\frac{\partial \phi_{\bar{r},\bar{h},\bar{y}'',\lambda}}{\partial y_i}$ ($i=4,5,\cdots,N$).
It follows from \eqref{con'} that
\begin{equation}\label{con''}
  \sum\limits_{l=2}^Nc_l\sum\limits_{j=1}^k\int_{\mathbb{R}^N}\Big
  (Z_{x_j^+,\lambda}^{m^*-2}Z_{j,l}^++Z_{x_j^-,\lambda}^{m^*-2}Z_{j,l}^-\Big)vdy=o(k\lambda^2)\sum\limits_{l=3}^N|c_l|+o(k\lambda^{1-\beta_1}|c_2|),
\end{equation}
also holds for $v=\langle y, \nabla Z_{\bar{r},\bar{h},\bar{y}'',\lambda}\rangle$ and $v=\frac{\partial Z_{\bar{r},\bar{h},\bar{y}'',\lambda}}{\partial y_i}$ ($i=4,5,\cdots,N$).

Since
\begin{equation*}
  \langle y, \nabla Z_{\bar{r},\bar{h},\bar{y}'',\lambda}\rangle=\langle y', \nabla _{y'}Z_{\bar{r},\bar{h},\bar{y}'',\lambda}\rangle+\langle y'', \nabla _{y''}Z_{\bar{r},\bar{h},\bar{y}'',\lambda}\rangle,
\end{equation*}
we obtain
\begin{align}\label{ob1}
  &\sum\limits_{l=2}^Nc_l\sum\limits_{j=1}^k\int_{\mathbb{R}^N}\Big
  (Z_{x_j^+,\lambda}^{m^*-2}Z_{j,l}^++Z_{x_j^-,\lambda}^{m^*-2}Z_{j,l}^-\Big)\langle y, \nabla Z_{\bar{r},\bar{h},\bar{y}'',\lambda}\rangle dy \nonumber\\
  =&c_3\sum\limits_{j=1}^k\int_{\mathbb{R}^N}\Big
  (Z_{x_j^+,\lambda}^{m^*-2}Z_{j,3}^++Z_{x_j^-,\lambda}^{m^*-2}Z_{j,3}^-\Big)\langle y', \nabla_{y'} Z_{\bar{r},\bar{h},\bar{y}'',\lambda}\rangle dy+o(k\lambda^2)\sum\limits_{l=4}^N|c_l|+o(k\lambda^{1-\beta_1}|c_2|),
\end{align}
and
\begin{align}\label{ob2}
  &\sum\limits_{l=2}^Nc_l\sum\limits_{j=1}^k\int_{\mathbb{R}^N}\Big
  (Z_{x_j^+,\lambda}^{m^*-2}Z_{j,l}^++Z_{x_j^-,\lambda}^{m^*-2}Z_{j,l}^-\Big)\frac{\partial Z_{\bar{r},\bar{h},\bar{y}'',\lambda}}{\partial y_i} dy \nonumber\\
  =&c_i\sum\limits_{j=1}^k\int_{\mathbb{R}^N}\Big
  (Z_{x_j^+,\lambda}^{m^*-2}Z_{j,i}^++Z_{x_j^-,\lambda}^{m^*-2}Z_{j,i}^-\Big)\frac{\partial Z_{\bar{r},\bar{h},\bar{y}'',\lambda}}{\partial y_i} dy+o(k\lambda^2)\sum\limits_{l\neq2,i}|c_l|+o(k\lambda^{1-\beta_1}|c_2|),
\end{align}
for $i=4,5,\cdots,N$.

Combining \eqref{con''}, \eqref{ob1} and \eqref{ob2}, we are led to
\begin{align*}
  c_3\sum\limits_{j=1}^k\int_{\mathbb{R}^N}\Big
  (Z_{x_j^+,\lambda}^{m^*-2}Z_{j,3}^++Z_{x_j^-,\lambda}^{m^*-2}Z_{j,3}^-\Big)\langle y', \nabla_{y'} Z_{\bar{r},\bar{h},\bar{y}'',\lambda}\rangle dy=o(k\lambda^2)\sum\limits_{l=4}^N|c_l|+o(k\lambda^{1-\beta_1}|c_2|),
\end{align*}
and
\begin{align*}
  c_i\sum\limits_{j=1}^k\int_{\mathbb{R}^N}\Big
  (Z_{x_j^+,\lambda}^{m^*-2}Z_{j,i}^++Z_{x_j^-,\lambda}^{m^*-2}Z_{j,i}^-\Big)\frac{\partial Z_{\bar{r},\bar{h},\bar{y}'',\lambda}}{\partial y_i} dy=o(k\lambda^2)\sum\limits_{l=3}^N|c_l|+o(k\lambda^{1-\beta_1}|c_2|),\quad i=4,5,\cdots,N,
\end{align*}
which together with \eqref{pr1} and \eqref{pr2} yields
\begin{equation*}
  c_l=o\Big(\frac{|c_2|}{\lambda^{1+\beta_1}}\Big),\quad l=3,4,\cdots,N.
\end{equation*}

On the other hand, it follows from \eqref{con'} and \eqref{pr3} that
\begin{align*}
  0=&\sum\limits_{l=2}^Nc_l\sum\limits_{j=1}^k\int_{\mathbb{R}^N}\Big
  (Z_{x_j^+,\lambda}^{m^*-2}Z_{j,l}^++Z_{x_j^-,\lambda}^{m^*-2}Z_{j,l}^-\Big)\frac{\partial Z_{\bar{r},\bar{h},\bar{y}'',\lambda}}{\partial \lambda}dy\\
  =&c_2\sum\limits_{j=1}^k\int_{\mathbb{R}^N}\Big
  (Z_{x_j^+,\lambda}^{m^*-2}Z_{j,2}^++Z_{x_j^-,\lambda}^{m^*-2}Z_{j,2}^-\Big)\frac{\partial Z_{\bar{r},\bar{h},\bar{y}'',\lambda}}{\partial \lambda}dy+o\Big(\frac{k|c_2|}{\lambda^{2\beta_1}}\Big)\\
  =&\frac{2k }{\lambda^{2\beta_1}} (a_3+o(1))c_2,
\end{align*}
which implies that $c_2=0$. The proof is complete.
\end{proof}

\begin{lemma}\label{ener1}
We have
\begin{align*}
  &\int_{\mathbb{R}^N}\big((-\Delta)^mu_k+V(r,y'')u_k-(u_k)_+^{m^*-1}\big)\frac{\partial Z_{\bar{r},\bar{h},\bar{y}'',\lambda}}{\partial \lambda} dy\\
  =&2k\bigg(-\frac{B_1}{\lambda^{2m+1}}V(\bar{r},\bar{y}'')+\sum\limits_{j=2}^k\frac{B_2}{\lambda^{N-2m+1}|x_j^+-x_1^+|^{N-2m}}+
  \sum\limits_{j=1}^k\frac{B_2}{\lambda^{N-2m+1}|x_j^--x_1^+|^{N-2m}}+O\Big(\frac{1}{\lambda^{2m+1+\varepsilon}}\Big)\bigg)\\
  =&2k\bigg(-\frac{B_1}{\lambda^{2m+1}}V(\bar{r},\bar{y}'')+\frac{B_3k^{N-2m}}{\lambda^{N-2m+1}(\sqrt{1-\bar{h}^2})^{N-2m}}+
  \frac{B_4 k}{\lambda^{N-2m+1}\bar{h}^{N-2m-1}\sqrt{1-\bar{h}^2}}+O\Big(\frac{1}{\lambda^{2m+1+\varepsilon}}\Big)\bigg),
\end{align*}
where $B_1$, $B_2$ are given in Lemma \ref{ener2}, and $B_3$, $B_4$ are two positive constants.
\end{lemma}

\begin{proof}
By symmetry, we have
\begin{align*}
  &\int_{\mathbb{R}^N}\big((-\Delta)^mu_k+V(r,y'')u_k-(u_k)_+^{m^*-1}\big)\frac{\partial Z_{\bar{r},\bar{h},\bar{y}'',\lambda}}{\partial \lambda} dy\\
  =&\Big\langle I'(Z_{\bar{r},\bar{h},\bar{y}'',\lambda}),\frac{\partial Z_{\bar{r},\bar{h},\bar{y}'',\lambda}}{\partial \lambda}\Big\rangle+2k
  \Big\langle (-\Delta)^m \phi+V(r,y'')\phi-(m^*-1)Z_{\bar{r},\bar{h},\bar{y}'',\lambda}^{m^*-2}\phi,\frac{\partial Z_{x_1^+,\lambda}}{\partial \lambda}\Big\rangle\\
  &-\int_{\mathbb{R}^N}\Big((Z_{\bar{r},\bar{h},\bar{y}'',\lambda}+\phi)_+^{m^*-1}-Z_{\bar{r},\bar{h},\bar{y}'',\lambda}^{m^*-1}-(m^*-1)Z_{\bar{r},\bar{h},\bar{y}'',\lambda}^{m^*-2}\phi\Big)\frac{\partial Z_{\bar{r},\bar{h},\bar{y}'',\lambda}}{\partial \lambda}  dy\\
  :=&\Big\langle I'(Z_{\bar{r},\bar{h},\bar{y}'',\lambda}),\frac{\partial Z_{\bar{r},\bar{h},\bar{y}'',\lambda}}{\partial \lambda}\Big\rangle+2kI_1-I_2.
\end{align*}
By \eqref{toener1} and \eqref{toener2}, we have
\begin{equation*}
 | I_1|=O\Big(\frac{\|\phi\|_*}{\lambda^{m+1+\varepsilon}}\Big)=O\Big(\frac{1}{\lambda^{2m+1+\varepsilon}}\Big).
\end{equation*}
If $N\geq 6m$,
we have
\begin{align*}
  |I_2|\leq & C\int_{\mathbb{R}^N}Z_{\bar{r},\bar{h},\bar{y}'',\lambda}^{m^*-3}\phi^2\Big|\frac{\partial Z_{\bar{r},\bar{h},\bar{y}'',\lambda}}{\partial \lambda}\Big|  dy\leq \frac{C}{\lambda^{\beta_1}}\int_{\mathbb{R}^N}Z_{\bar{r},\bar{h},\bar{y}'',\lambda}^{m^*-2}\phi^2 dy\nonumber\\
  \leq & \frac{C\lambda^N\|\phi\|_*^2}{\lambda^{\beta_1}}\int_{\mathbb{R}^N} \bigg(\sum\limits_{j=1}^k\Big(\frac{1}{(1+\lambda|y-x_j^+|)^{{N-2m}}}+\frac{1}{(1+\lambda|y-x_j^-|)^{{N-2m}}}\Big)\bigg)^{m^*-2} \nonumber\\&\times
  \bigg(\sum\limits_{j=1}^k\Big(\frac{1}{(1+\lambda|y-x_j^+|)^{\frac{N-2m}{2}+\tau}}+\frac{1}{(1+\lambda|y-x_j^-|)^{\frac{N-2m}{2}+\tau}}\Big)\bigg)^{2}dy \nonumber\\
  \leq & \frac{C\lambda^N\|\phi\|_*^2}{\lambda^{\beta_1}}\int_{\mathbb{R}^N} \sum\limits_{j=1}^k\Big(\frac{1}{(1+\lambda|y-x_j^+|)^{{4m}}}+\frac{1}{(1+\lambda|y-x_j^-|)^{{4m}}}\Big) \nonumber\\&\times
  \sum\limits_{j=1}^k\Big(\frac{1}{(1+\lambda|y-x_j^+|)^{{N-2m}+2\tau}}+\frac{1}{(1+\lambda|y-x_j^-|)^{{N-2m}+2\tau}}\Big)dy
\nonumber\\ \leq & \frac{Ck\|\phi\|_*^2}{\lambda^{\beta_1}}=O\Big(\frac{k}{\lambda^{2m+1+\varepsilon}}\Big).
\end{align*}
Similarly, we can prove the case $4m+1 < N \leq 6m - 1$.
Combining Lemmas \ref{AppA6} and \ref{ener2}, we finish the proof.
\end{proof}

By Lemma \ref{AppA8}, we have
\begin{equation}\label{con11}
  \int_{D_\varrho}(-\Delta)^mu_k\langle y, \nabla u_k\rangle dy=\frac{1}{2}\int_{\partial D_\varrho}f_m(u_k,u_k) d\sigma-\frac{N-2m}{2}\int_{D_\varrho}u_k(-\Delta)^mu_kdy.
\end{equation}
Integrating by parts, we obtain
\begin{equation}\label{con12}
   \int_{D_\varrho}V(y)u_k\langle y, \nabla u_k\rangle dy= \frac{1}{2}\int_{\partial D_\varrho}\varrho V(y)u_k^2d\sigma-\frac{1}{2} \int_{D_\varrho}\langle y, \nabla V(y)\rangle u_k^2dy- \frac{N}{2}\int_{D_\varrho}V(y)u_k^2 dy,
\end{equation}
and
\begin{equation}\label{con13}
   \int_{D_\varrho}(u_k)^{m^*-1}_+\langle y, \nabla u_k\rangle dy= \frac{1}{m^*}\int_{\partial D_\varrho}\varrho (u_k)_+^{m^*}d\sigma- \frac{N}{m^*}\int_{D_\varrho}(u_k)_+^{m^*} dy.
\end{equation}
Combining \eqref{con11}, \eqref{con12} and \eqref{con13}, we know that \eqref{con1} is equivalent to
\begin{align}\label{trans1}
  &\frac{2m-N}{2}\int_{D_\varrho}u_k(-\Delta)^mu_kdy-\frac{1}{2} \int_{D_\varrho}\langle y, \nabla V(y)\rangle u_k^2dy- \frac{N}{2}\int_{D_\varrho}V(y)u_k^2 dy+\frac{N}{m^*}\int_{D_\varrho}(u_k)_+^{m^*} dy \nonumber\\
  =&O\bigg(\int_{\partial D_\varrho}\Big(\phi^2+|\phi|^{m^*}+\sum\limits_{j=1}^{2m-1}|\nabla^j \phi ||\nabla ^{2m-j}\phi|+\sum\limits_{j=0}^{2m-1}|\nabla^j \phi ||\nabla ^{2m-j-1}\phi|\Big)d\sigma\bigg),
\end{align}
since $u_k=\phi$ on $\partial D_\varrho$.

Similarly, for $i=4,5,\cdots,N$, Lemma \ref{AppA9} shows
\begin{equation}\label{con21}
  \int_{D_\varrho}(-\Delta)^mu_k\frac{\partial u_k}{\partial y_i} dy=\frac{1}{2}\int_{\partial D_\varrho}g_{m,i}(u_k,u_k) d\sigma.
\end{equation}
Integrating by parts, we have
\begin{equation}\label{con22}
   \int_{D_\varrho}V(r,y'')u_k\frac{\partial u_k}{\partial y_i} dy= \frac{1}{2}\int_{\partial D_\varrho} V(r,y'')u_k^2 \nu_i d\sigma-\frac{1}{2} \int_{D_\varrho}\frac{\partial V(r,y'')}{\partial y_i}u_k^2dy,
\end{equation}
and
\begin{equation}\label{con23}
   \int_{D_\varrho}(u_k)^{m^*-1}_+\frac{\partial u_k}{\partial y_i} dy= \frac{1}{m^*}\int_{\partial D_\varrho} (u_k)_+^{m^*}\nu_id\sigma,
\end{equation}
 where $\nu=(\nu_1,\nu_2,\cdots,\nu_N)$ denotes the outward unit normal vector of $\partial D_\varrho$.
Combining \eqref{con21}, \eqref{con22} and \eqref{con23}, we know that \eqref{con2} is equivalent to
\begin{equation}\label{trans2}
  \int_{D_\varrho}\frac{\partial V(r,y'')}{\partial y_i}u_k^2dy=O\bigg(\int_{\partial D_\varrho}\Big(\phi^2+|\phi|^{m^*}+\sum\limits_{j=1}^{2m-1}|\nabla^j \phi ||\nabla ^{2m-j}\phi|\Big)d\sigma\bigg),\quad i=4,5,\cdots,N.
\end{equation}

Multiplying \eqref{pp} by $u_k$ and integrating in $D_\varrho$, we obtain
\begin{align*}
  \int_{D_\varrho}u_k(-\Delta)^mu_kdy+\int_{D_\varrho}V(y)u_k^2 dy
  =\int_{D_\varrho}(u_k)_+^{m^*} dy+\sum\limits_{l=2}^Nc_l\sum\limits_{j=1}^k\int_{D_\varrho}\Big
  (Z_{x_j^+,\lambda}^{m^*-2}Z_{j,l}^++Z_{x_j^-,\lambda}^{m^*-2}Z_{j,l}^-\Big)Z_{\bar{r},\bar{h},\bar{y}'',\lambda}dy.
\end{align*}
Thus, \eqref{trans1} can be reduced to
\begin{align}\label{trans1'}
  &m\int_{D_\varrho}V(y)u_k^2 dy+\frac{1}{2} \int_{D_\varrho}\langle y, \nabla V(y)\rangle u_k^2dy\nonumber\\
  =&\frac{2m-N}{2}\sum\limits_{l=2}^Nc_l\sum\limits_{j=1}^k\int_{D_\varrho}\Big
  (Z_{x_j^+,\lambda}^{m^*-2}Z_{j,l}^++Z_{x_j^-,\lambda}^{m^*-2}Z_{j,l}^-\Big)Z_{\bar{r},\bar{h},\bar{y}'',\lambda}dy \nonumber\\
  &+O\bigg(\int_{\partial D_\varrho}\Big(\phi^2+|\phi|^{m^*}+\sum\limits_{j=1}^{2m-1}|\nabla^j \phi ||\nabla ^{2m-j}\phi|+\sum\limits_{j=0}^{2m-1}|\nabla^j \phi ||\nabla ^{2m-j-1}\phi|\Big)d\sigma\bigg).
\end{align}
Using \eqref{trans2}, we can rewrite \eqref{trans1'} as
\begin{align}\label{trans1''}
  &m\int_{D_\varrho}V(y)u_k^2 dy+\frac{1}{2} \int_{D_\varrho}r \frac{\partial V(r,y'')}{\partial r} u_k^2dy\nonumber\\
  =&\frac{2m-N}{2}\sum\limits_{l=2}^Nc_l\sum\limits_{j=1}^k\int_{D_\varrho}\Big
  (Z_{x_j^+,\lambda}^{m^*-2}Z_{j,l}^++Z_{x_j^-,\lambda}^{m^*-2}Z_{j,l}^-\Big)Z_{\bar{r},\bar{h},\bar{y}'',\lambda}dy \nonumber\\
  &+O\bigg(\int_{\partial D_\varrho}\Big(\phi^2+|\phi|^{m^*}+\sum\limits_{j=1}^{2m-1}|\nabla^j \phi ||\nabla ^{2m-j}\phi|+\sum\limits_{j=0}^{2m-1}|\nabla^j \phi ||\nabla ^{2m-j-1}\phi|\Big)d\sigma\bigg).
\end{align}
A direct computation gives
\begin{equation*}
  \sum\limits_{j=1}^k\int_{D_\varrho}\Big
  (Z_{x_j^+,\lambda}^{m^*-2}Z_{j,l}^++Z_{x_j^-,\lambda}^{m^*-2}Z_{j,l}^-\Big)Z_{\bar{r},\bar{h},\bar{y}'',\lambda}dy=O\Big(\frac{k\lambda^{\eta_l}}{\lambda^{2m}}\Big),
\end{equation*}
this with Proposition \ref{fixed} yields
\begin{equation*}
  \sum\limits_{l=2}^Nc_l\sum\limits_{j=1}^k\int_{D_\varrho}\Big
  (Z_{x_j^+,\lambda}^{m^*-2}Z_{j,l}^++Z_{x_j^-,\lambda}^{m^*-2}Z_{j,l}^-\Big)Z_{\bar{r},\bar{h},\bar{y}'',\lambda}dy=O\Big(\frac{k}{\lambda^{3m+\frac{1-\beta_1}{2}+\varepsilon}}\Big)=
  o\Big(\frac{k}{\lambda^{2m}}\Big).
\end{equation*}
Therefore, \eqref{trans1''} is equivalent to
\begin{align}\label{trans1'''}
  & \int_{D_\varrho}\frac{1}{2r^{2m-1}} \frac{\partial\big(r^{2m} V(r,y'')\big)}{\partial r} u_k^2dy\nonumber\\
  =& o\Big(\frac{k}{\lambda^{2m}}\Big)
  +O\bigg(\int_{\partial D_\varrho}\Big(\phi^2+|\phi|^{m^*}+\sum\limits_{j=1}^{2m-1}|\nabla^j \phi ||\nabla ^{2m-j}\phi|+\sum\limits_{j=0}^{2m-1}|\nabla^j \phi ||\nabla ^{2m-j-1}\phi|\Big)d\sigma\bigg).
\end{align}

First, we estimate \eqref{trans2} and \eqref{trans1'''} from the right hand, and it is sufficient to estimate
\begin{equation*}
  \int_{D_{4\delta}\backslash D_{3\delta}}\Big(\phi^2+|\phi|^{m^*}+\sum\limits_{j=1}^{2m-1}|\nabla^j \phi ||\nabla ^{2m-j}\phi|+\sum\limits_{j=0}^{2m-1}|\nabla^j \phi ||\nabla ^{2m-j-1}\phi|\Big)dy.
\end{equation*}
We first prove
\begin{lemma}\label{fi}
It holds
\begin{equation*}
  \int_{\mathbb{R}^N}\Big(\phi(-\Delta)^m \phi+V(r,y'')\phi^2\Big)dy=O\Big(\frac{k}{\lambda^{2m+1-\beta_1+\varepsilon}}\Big).
\end{equation*}
\end{lemma}
\begin{proof}
Multiplying \eqref{pp} by $\phi$ and integrating in $\mathbb{R}^N$, we have
\begin{align*}
  &\int_{\mathbb{R}^N}\Big(\phi(-\Delta)^m \phi+V(r,y'')\phi^2\Big)dy\\=&\int_{\mathbb{R}^N}\Big((Z_{\bar{r},\bar{h},\bar{y}'',\lambda}+\phi)_+^{m^*-1}-V(r,y'')Z_{\bar{r},\bar{h},\bar{y}'',\lambda}-
  (-\Delta)^mZ_{\bar{r},\bar{h},\bar{y}'',\lambda}\Big)\phi dy\\
  =&\int_{\mathbb{R}^N}\Big((Z_{\bar{r},\bar{h},\bar{y}'',\lambda}+\phi)_+^{m^*-1}-Z_{\bar{r},\bar{h},\bar{y}'',\lambda}^{m^*-1}\Big)\phi dy-\int_{\mathbb{R}^N}V(r,y'')Z_{\bar{r},\bar{h},\bar{y}'',\lambda}
  \phi dy\\&+\int_{\mathbb{R}^N}\Big(Z_{\bar{r},\bar{h},\bar{y}'',\lambda}^{m^*-1}-(-\Delta)^mZ_{\bar{r},\bar{h},\bar{y}'',\lambda}\Big)\phi dy\\
  :=&I_1-I_2+I_3.
\end{align*}
If $N\geq 6m$, by \eqref{new1}, we have
\begin{align*}
  |I_1|\leq &C\int_{\mathbb{R}^N}Z_{\bar{r},\bar{h},\bar{y}'',\lambda}^{m^*-2}\phi^2 dy+C\int_{\mathbb{R}^N}|\phi|^{m^*} dy\\
  \leq & C\lambda^N(\|\phi\|_*^2+\|\phi\|_*^{m^*})\int_{\mathbb{R}^N}
  \bigg(\sum\limits_{j=1}^k\Big(\frac{1}{(1+\lambda|y-x_j^+|)^{\frac{N-2m}{2}+\tau}}+\frac{1}{(1+\lambda|y-x_j^-|)^{\frac{N-2m}{2}+\tau}}\Big)\bigg)^{m^*}dy\\
  \leq & C\lambda^N(\|\phi\|_*^2+\|\phi\|_*^{m^*})\int_{\mathbb{R}^N}\sum\limits_{j=1}^k\Big(\frac{1}{(1+\lambda|y-x_j^+|)^{\frac{N+2m}{2}+\tau}}+
  \frac{1}{(1+\lambda|y-x_j^-|)^{\frac{N+2m}{2}+\tau}}\Big)\\
  & \times \sum\limits_{j=1}^k\Big(\frac{1}{(1+\lambda|y-x_j^+|)^{\frac{N-2m}{2}+\tau}}+
  \frac{1}{(1+\lambda|y-x_j^-|)^{\frac{N-2m}{2}+\tau}}\Big)dy\\
  \leq & Ck\|\phi\|_*^2=O\Big(\frac{k}{\lambda^{2m+1-\beta_1+\varepsilon}}\Big).
\end{align*}
Similarly, if $4m+1 < N \leq 6m - 1$, we can prove that $I_1=O\big(\frac{k}{\lambda^{2m+1-\beta_1+\varepsilon}}\big)$.

For $I_2$, by symmetry and \eqref{err3}, we can deduce
\begin{align*}
  |I_2|\leq& C\|\phi\|_* \big(\frac{1}{\lambda}\big)^{\frac{2m+1-\beta_1}{2}+\varepsilon}\lambda^{N} \int_{\mathbb{R}^N}\sum\limits_{j=1}^k\Big(\frac{1}{(1+\lambda|y-x_j^+|)^{\frac{N+2m}{2}+\tau}}+
  \frac{1}{(1+\lambda|y-x_j^-|)^{\frac{N+2m}{2}+\tau}}\Big)\\
  & \times \sum\limits_{j=1}^k\Big(\frac{1}{(1+\lambda|y-x_j^+|)^{\frac{N-2m}{2}+\tau}}+
  \frac{1}{(1+\lambda|y-x_j^-|)^{\frac{N-2m}{2}+\tau}}\Big)dy\\
  \leq &Ck\|\phi\|_* \big(\frac{1}{\lambda}\big)^{\frac{2m+1-\beta_1}{2}+\varepsilon}=O\Big(\frac{k}{\lambda^{2m+1-\beta_1+\varepsilon}}\Big).
\end{align*}

For $I_3$, by \eqref{err1}, \eqref{err2}, \eqref{err4} and \eqref{err5}, we  obtain
\begin{align*}
  |I_3|\leq &C \|\phi\|_* \big(\frac{1}{\lambda}\big)^{\frac{2m+1-\beta_1}{2}+\varepsilon}\lambda^{N}\int_{\mathbb{R}^N}\Big(\sum\limits_{j=1}^k\frac{1}{(1+\lambda|y-x_j^+|)^{\frac{N+2m}{2}+\tau}}+\sum\limits_{j=1}^k\frac{1}{(1+\lambda|y-x_j^-|)^{
  \frac{N+2m}{2}+\tau}}\Big)\\
& \times \sum\limits_{j=1}^k\Big(\frac{1}{(1+\lambda|y-x_j^+|)^{\frac{N-2m}{2}+\tau}}+
  \frac{1}{(1+\lambda|y-x_j^-|)^{\frac{N-2m}{2}+\tau}}\Big)dy\\
  \leq &Ck\|\phi\|_* \big(\frac{1}{\lambda}\big)^{\frac{2m+1-\beta_1}{2}+\varepsilon}=O\Big(\frac{k}{\lambda^{2m+1-\beta_1+\varepsilon}}\Big).
\end{align*}
This completes the proof.
\end{proof}

Next, we prove
\begin{lemma}\label{se}
It holds
\begin{equation*}
  \int_{D_{4\delta}\backslash D_{3\delta}}\big(\phi^2+|\phi|^{m^*}\big)dy=O\Big(\frac{k}{\lambda^{2m+1-\beta_1+\varepsilon}}\Big).
\end{equation*}
\end{lemma}
\begin{proof}
We have
\begin{align*}
  \int_{D_{4\delta}\backslash D_{3\delta}}\phi^2dy\leq &\frac{C\|\phi\|^2_*}{\lambda^{2m}}\lambda^N\int_{D_{4\delta}\backslash D_{3\delta}} \bigg(\sum\limits_{j=1}^k\Big(\frac{1}{(1+\lambda|y-x_j^+|)^{\frac{N-2m}{2}+\tau}}+\frac{1}{(1+\lambda|y-x_j^-|)^{\frac{N-2m}{2}+\tau}}\Big)\bigg)^{2}dy\\
  \leq &\frac{Ck\|\phi\|^2_*}{\lambda^{2m}}\lambda^N\int_{D_{4\delta}} \Big(\frac{1}{(1+\lambda|y-x_1^+|)^{{N-2m}+2\tau}}
  +\sum\limits_{j=2}^k\frac{1}{(\lambda|x_j^+-x_1^+|)^\tau}
  \frac{1}{(1+\lambda|y-x_j^+|)^{{N-2m}+\tau}}\\&+\sum\limits_{j=1}^k\frac{1}{(\lambda|x_j^--x_1^+|)^\tau}\frac{1}{(1+\lambda|y-x_j^-|)^{{N-2m}+\tau}}\Big)dy\\
  \leq &\frac{Ck\|\phi\|^2_*}{\lambda^{2m}}\lambda^N\int_{D_{4\delta}} \frac{1}{(1+\lambda|y-x_1^+|)^{{N-2m}+\tau}}dy\\
  \leq &\frac{Ck\|\phi\|^2_*}{\lambda^{\tau}}=O\Big(\frac{k}{\lambda^{2m+1+\tau-\beta_1+\varepsilon}}\Big)=O\Big(\frac{k}{\lambda^{2m+1-\beta_1+\varepsilon}}\Big).
\end{align*}
Since $m^*>2$, by \eqref{new1}, we obtain
\begin{align*}
  \int_{D_{4\delta}\backslash D_{3\delta}}|\phi|^{m^*}dy\leq &{C\|\phi\|^{m^*}_*}\lambda^N\int_{\mathbb{R}^N} \bigg(\sum\limits_{j=1}^k\Big(\frac{1}{(1+\lambda|y-x_j^+|)^{\frac{N-2m}{2}+\tau}}+\frac{1}{(1+\lambda|y-x_j^-|)^{\frac{N-2m}{2}+\tau}}\Big)\bigg)^{m^*}dy\\
  \leq &{C\|\phi\|^{m^*}_*}\lambda^N\int_{\mathbb{R}^N} \sum\limits_{j=1}^k\Big(\frac{1}{(1+\lambda|y-x_j^+|)^{\frac{N+2m}{2}+\tau}}+\frac{1}{(1+\lambda|y-x_j^-|)^{\frac{N+2m}{2}+\tau}}\Big)\\
  &\times \sum\limits_{j=1}^k\Big(\frac{1}{(1+\lambda|y-x_j^+|)^{\frac{N-2m}{2}+\tau}}+\frac{1}{(1+\lambda|y-x_j^-|)^{\frac{N-2m}{2}+\tau}}\Big)dy\\
  \leq &{Ck\|\phi\|^{m^*}_*}=O\Big(\frac{k}{\lambda^{2m+1-\beta_1+\varepsilon}}\Big).
\end{align*}
This finishes the proof.
\end{proof}

We also have the following point wise estimate for the error term.

\begin{lemma}\label{thi}
It holds
\begin{equation*}
  \|\phi\|_{C^{2m-1}(D_{4\delta}\backslash D_{3\delta})}\leq \frac{C}{\lambda^{m+\tau+\varepsilon}}.
\end{equation*}
\end{lemma}
\begin{proof}
By Proposition \ref{fixed}, for any $y\in D_{4\delta}\backslash D_{3\delta}$, we obtain
\begin{equation*}
  |\phi(y)|\leq \|\phi\|_*\lambda^{\frac{N-2m}{2}}\sum\limits_{j=1}^k\Big(\frac{1}{(1+\lambda|y-x_j^+|)^{\frac{N-2m}{2}+\tau}}+\frac{1}{(1+\lambda|y-x_j^-|)^{\frac{N-2m}{2}+\tau}}\Big)
  \leq \frac{C}{\lambda^{m+\tau+\varepsilon}}.
\end{equation*}

On the other hand, we know that $\phi$ satisfies
\begin{equation*}
  (-\Delta)^m \phi+V(r,y'')\phi-(m^*-1)Z_{\bar{r},\bar{h},\bar{y}'',\lambda}^{m^*-2}\phi
  =N(\phi)+E_k,\quad \text{in $\mathbb{R}^N$}.
\end{equation*}
By using the $L^p$ estimates, we deduce that for any $p>1$,
\begin{align*}
  \|\phi\|_{W^{2m,p}(D_{4\delta}\backslash D_{3\delta})}\leq& C\|\phi\|_{L^{\infty}(D_{4\delta}\backslash D_{3\delta})}+C\|
  (m^*-1)Z_{\bar{r},\bar{h},\bar{y}'',\lambda}^{m^*-2}\phi-V(r,y'')\phi
  +N(\phi)+E_k \|_{L^{\infty}(D_{4\delta}\backslash D_{3\delta})}\\
  \leq &C\|\phi\|_{L^{\infty}(D_{4\delta}\backslash D_{3\delta})}+\frac{C(\|N(\phi)\|_{**}+\|E_k\|_{**})}{\lambda^\tau}.
\end{align*}
The result follows from Lemmas \ref{non} and \ref{err}.
\end{proof}

From Lemmas \ref{fi}, \ref{se} and \ref{thi}, we know
\begin{equation*}
  \int_{D_{4\delta}\backslash D_{3\delta}}\Big(\phi^2+|\phi|^{m^*}+\sum\limits_{j=1}^{2m-1}|\nabla^j \phi ||\nabla ^{2m-j}\phi|+\sum\limits_{j=0}^{2m-1}|\nabla^j \phi ||\nabla ^{2m-j-1}\phi|\Big)dy=O\Big(\frac{k}{\lambda^{2m+1-\beta_1+\varepsilon}}\Big)=o\Big(\frac{k}{\lambda^{2m}}\Big).
\end{equation*}
Thus, there exists $\varrho\in (3\delta,4\delta)$ such that
\begin{equation}\label{trans3}
  \int_{\partial D_{\varrho}}\Big(\phi^2+|\phi|^{m^*}+\sum\limits_{j=1}^{2m-1}|\nabla^j \phi ||\nabla ^{2m-j}\phi|+\sum\limits_{j=0}^{2m-1}|\nabla^j \phi ||\nabla ^{2m-j-1}\phi|\Big)d\sigma=o\Big(\frac{k}{\lambda^{2m}}\Big).
\end{equation}

Conversely, we need to estimate \eqref{trans2} and \eqref{trans1'''} from the left hand, and we have the following lemma.
\begin{lemma}\label{converse}
For any function $h(r,y'')\in C^1(\mathbb{R}^{N-2},\mathbb{R})$, there holds
\begin{equation*}
  \int_{ D_{\varrho}}h(r,y'')u_k^2dy=2k\Big(\frac{1}{\lambda^{2m}}h(\bar{r},\bar{y}'')\int_{\mathbb{R}^N}U_{0,1}^2dy+o\big(\frac{1}{\lambda^{2m}}\big)\Big).
\end{equation*}
\end{lemma}
\begin{proof}
Since $u_k=Z_{\bar{r},\bar{h},\bar{y}'',\lambda}+\phi$, we have
\begin{equation*}
  \int_{ D_{\varrho}}h(r,y'')u_k^2dy=\int_{ D_{\varrho}}h(r,y'')Z_{\bar{r},\bar{h},\bar{y}'',\lambda}^2dy+2\int_{ D_{\varrho}}h(r,y'')Z_{\bar{r},\bar{h},\bar{y}'',\lambda} \phi dy+\int_{ D_{\varrho}}h(r,y'')\phi^2dy.
\end{equation*}
For the first term, a direct computation leads to
\begin{equation*}
  \int_{ D_{\varrho}}h(r,y'')Z_{\bar{r},\bar{h},\bar{y}'',\lambda}^2dy=2k\Big(\frac{1}{\lambda^{2m}}h(\bar{r},\bar{y}'')\int_{\mathbb{R}^N}U_{0,1}^2dy+o\big(\frac{1}{\lambda^{2m}}\big)\Big).
\end{equation*}

For the second term, by symmetry and \eqref{err3}, we obtain
\begin{align*}
  &\Big|\int_{ D_{\varrho}}h(r,y'')Z_{\bar{r},\bar{h},\bar{y}'',\lambda} \phi dy\Big|\\
  \leq &C\|\phi\|_*\big(\frac{1}{\lambda}\big)^{\frac{2m+1-\beta_1}{2}+\varepsilon} \lambda^N\int_{ \mathbb{R}^N}\sum\limits_{j=1}^k\Big(\frac{1}{(1+\lambda|y-x_j^+|)^{\frac{N+2m}{2}+\tau}}+\frac{1}{(1+\lambda|y-x_j^-|)^{\frac{N+2m}{2}+\tau}}\Big)\\
  &\times \sum\limits_{j=1}^k\Big(\frac{1}{(1+\lambda|y-x_j^+|)^{\frac{N-2m}{2}+\tau}}+\frac{1}{(1+\lambda|y-x_j^-|)^{\frac{N-2m}{2}+\tau}}\Big)dy\\
  \leq &Ck\|\phi\|_*\big(\frac{1}{\lambda}\big)^{\frac{2m+1-\beta_1}{2}+\varepsilon}=O\Big(\frac{k}{\lambda^{2m+{1-\beta_1}+\varepsilon}}\Big)=o\Big(\frac{k}{\lambda^{2m}}\Big).
\end{align*}

For the third term, we have
\begin{align*}
  \Big|\int_{ D_{\varrho}}h(r,y'')\phi^2dy\Big|\leq& \frac{C\|\phi\|^2_*}{\lambda^{2m}}\lambda^N\int_{D_{4\delta}\backslash D_{3\delta}} \bigg(\sum\limits_{j=1}^k\Big(\frac{1}{(1+\lambda|y-x_j^+|)^{\frac{N-2m}{2}+\tau}}+\frac{1}{(1+\lambda|y-x_j^-|)^{\frac{N-2m}{2}+\tau}}\Big)\bigg)^{2}dy\\
  \leq &\frac{Ck\|\phi\|^2_*}{\lambda^{\tau}}=O\Big(\frac{k}{\lambda^{2m+1+\tau-\beta_1+\varepsilon}}\Big)=o\Big(\frac{k}{\lambda^{2m}}\Big).
\end{align*}
So we get the result.
\end{proof}

With these preliminaries at hand, now we will prove Theorem \ref{th1}.

\vspace{.3cm}

\noindent{\bf Proof of Theorem \ref{th1}.} Through the above discussion, applying \eqref{trans3} and Lemma \ref{converse} to \eqref{trans2} and \eqref{trans1'''}, we can see that \eqref{con1} and \eqref{con2} are equivalent to
\begin{equation*}
  2k\Big(\frac{1}{\lambda^{2m}}\frac{1}{2\bar{r}^{2m-1}} \frac{\partial\big(\bar{r}^{2m} V(\bar{r},\bar{y}'')\big)}{\partial \bar{r}} \int_{\mathbb{R}^N}U_{0,1}^2dy+o\big(\frac{1}{\lambda^{2m}}\big)\Big)=o\Big(\frac{k}{\lambda^{2m}}\Big),
\end{equation*}
and
\begin{equation*}
  2k\Big(\frac{1}{\lambda^{2m}}\frac{\partial V(\bar{r},\bar{y}'')}{\partial \bar{y}_i}\int_{\mathbb{R}^N}U_{0,1}^2dy+o\big(\frac{1}{\lambda^{2m}}\big)\Big)=o\Big(\frac{k}{\lambda^{2m}}\Big),\quad i=4,5,\cdots,N.
\end{equation*}
Therefore, the equations to determine $(\bar{r},\bar{y}'')$ are
\begin{equation}\label{de1}
  \frac{\partial\big(\bar{r}^{2m} V(\bar{r},\bar{y}'')\big)}{\partial \bar{r}} =o(1),
\end{equation}
and
\begin{equation}\label{de2}
  \frac{\partial\big(\bar{r}^{2m} V(\bar{r},\bar{y}'')\big)}{\partial \bar{y}_i}=o(1),\quad i=4,5,\cdots,N.
\end{equation}
Moreover, by Lemma \ref{ener1}, the equation to determine $\lambda$ is
\begin{equation}\label{dela}
  -\frac{B_1}{\lambda^{2m+1}}V(\bar{r},\bar{y}'')+\frac{B_3k^{N-2m}}{\lambda^{N-2m+1}(\sqrt{1-\bar{h}^2})^{N-2m}}+
  \frac{B_4 k}{\lambda^{N-2m+1}\bar{h}^{N-2m-1}\sqrt{1-\bar{h}^2}}=O\Big(\frac{1}{\lambda^{2m+1+\varepsilon}}\Big),
\end{equation}
where $B_1$, $B_3$, $B_4$ are positive constants.

Let $\lambda=tk^{\frac{N-2m}{N-4m-\alpha}}$ with $\alpha=N-4M-\iota$, $\iota$ is a small constant, then $t\in [L_0,L_1]$.
From
\eqref{dela}, we have
\begin{equation*}
  -\frac{B_1}{t^{2m+1}}V(\bar{r},\bar{y}'')+\frac{B_3M_1^{N-2m}}{t^{N-2m+1-\alpha}}=o(1),\quad t\in [L_0,L_1].
\end{equation*}

Define
\begin{equation*}
  F(t,\bar{r},\bar{y}'')=\Big(\nabla _{\bar{r},\bar{y}''}\big(\bar{r}^{2m}V(\bar{r},\bar{y}'')\big),-\frac{B_1}{t^{2m+1}}V(\bar{r},\bar{y}'')+\frac{B_3M_1^{N-2m}}{t^{N-2m+1-\alpha}}\Big).
\end{equation*}
Then, it holds
\begin{equation*}
  deg\Big(F(t,\bar{r},\bar{y}''),[L_0,L_1]\times B_\vartheta\big((r_0,y_0'')\big)\Big)=-deg\Big(\nabla _{\bar{r},\bar{y}''}\big(\bar{r}^{2m}V(\bar{r},\bar{y}'')\big),B_\vartheta\big((r_0,y_0'')\big)\Big)\neq0.
\end{equation*}
Hence, \eqref{de1}, \eqref{de2} and \eqref{dela} has a solution $t_k\in [L_0,L_1]$, $(\bar{r}_k,\bar{y}_k'')\in B_\vartheta\big((r_0,y_0'')\big)$.
\qed

\section{Proof of Theorem \ref{th2}}\label{four}
In this section, we give a brief proof of Theorem \ref{th2}. We define $\tau=\frac{N-4m}{N-2m}$.

\vspace{.3cm}

\noindent{\bf Proof of Theorem \ref{th2}.} We can verify that
\begin{equation}\label{same}
  \frac{k}{\lambda^\tau}=O(1),\quad \frac{k}{\lambda}=O\Big(\big(\frac{1}{\lambda}\big)^{\frac{2m}{N-2m}}\Big).
\end{equation}
Using \eqref{same} and Lemma \ref{AppA5}, we get the same conclusions for problems arising from the distance between points $\{x_j^+\}_{j=1}^k$ and $\{x_j^-\}_{j=1}^k$.

Moreover, by Lemma \ref{AppA4}, we have
\begin{equation}\label{same1}
 |Z_{j,2}^{\pm}|\leq C\lambda^{-\beta_2}Z_{x_j^\pm,\lambda},\quad |Z_{j,l}^{\pm}|\leq C\lambda Z_{x_j^\pm,\lambda},\quad l=3,4,\cdots,N,
\end{equation}
where $\beta_2=\frac{N-4m}{N-2m}$.

Using \eqref{same} and \eqref{same1}, with a similar step in the proof of Theorem \ref{th1} in Sections \ref{two} and \ref{three}, we know
that the
proof of Theorem \ref{th2} has the same reduction structure as that of Theorem \ref{th1} and $u_k$ is a solution of problem
\eqref{pro} if the following equalities hold:
\begin{equation}\label{de1'}
  \frac{\partial\big(\bar{r}^{2m} V(\bar{r},\bar{y}'')\big)}{\partial \bar{r}} =o(1),
\end{equation}
\begin{equation}\label{de2'}
  \frac{\partial\big(\bar{r}^{2m} V(\bar{r},\bar{y}'')\big)}{\partial \bar{y}_i}=o(1),\quad i=4,5,\cdots,N,
\end{equation}
\begin{equation}\label{dela'}
  -\frac{B_1}{\lambda^{2m+1}}V(\bar{r},\bar{y}'')+\frac{B_3k^{N-2m}}{\lambda^{N-2m+1}(\sqrt{1-\bar{h}^2})^{N-2m}}+
  \frac{B_4 k}{\lambda^{N-2m+1}\bar{h}^{N-2m-1}\sqrt{1-\bar{h}^2}}=O\Big(\frac{1}{\lambda^{2m+1+\varepsilon}}\Big).
\end{equation}

Let $\lambda=tk^{\frac{N-2m}{N-4m}}$, then $t\in [L_0',L_1']$.
Next, we discuss the main items in \eqref{dela'}.

{\bf Case 1.} If $\bar{h}\rightarrow A\in (0,1)$, then $(\lambda^{\frac{N-4m}{N-2m}}\bar{h})^{-1}\rightarrow0$ as $\lambda\rightarrow\infty$,
from
\eqref{dela'}, we have
\begin{equation*}
  -\frac{B_1}{t^{2m+1}}V(\bar{r},\bar{y}'')+\frac{B_3}{t^{N-2m+1}(\sqrt{1-A^2})^{N-2m}}=o(1),\quad t\in [L_0',L_1'].
\end{equation*}
Define
\begin{equation*}
  F(t,\bar{r},\bar{y}'')=\Big(\nabla _{\bar{r},\bar{y}''}\big(\bar{r}^{2m}V(\bar{r},\bar{y}'')\big),-\frac{B_1}{t^{2m+1}}V(\bar{r},\bar{y}'')+\frac{B_3}{t^{N-2m+1}(\sqrt{1-A^2})^{N-2m}}\Big).
\end{equation*}
Then, it holds
\begin{equation*}
  deg\Big(F(t,\bar{r},\bar{y}''),[L_0',L_1']\times B_\vartheta\big((r_0,y_0'')\big)\Big)=-deg\Big(\nabla _{\bar{r},\bar{y}''}\big(\bar{r}^{2m}V(\bar{r},\bar{y}'')\big),B_\vartheta\big((r_0,y_0'')\big)\Big)\neq0.
\end{equation*}
Hence, \eqref{de1'}, \eqref{de2'} and \eqref{dela'} has a solution $t_k\in [L_0',L_1']$, $(\bar{r}_k,\bar{y}_k'')\in B_\vartheta\big((r_0,y_0'')\big)$.

{\bf Case 2.} If $\bar{h}\rightarrow 0$ and $(\lambda^{\frac{N-4m}{N-2m}}\bar{h})^{-1}\rightarrow0$ as $\lambda\rightarrow\infty$,
from
\eqref{dela'}, we have
\begin{equation*}
  -\frac{B_1}{t^{2m+1}}V(\bar{r},\bar{y}'')+\frac{B_3}{t^{N-2m+1}}=o(1),\quad t\in [L_0',L_1'].
\end{equation*}
Define
\begin{equation*}
  F(t,\bar{r},\bar{y}'')=\Big(\nabla _{\bar{r},\bar{y}''}\big(\bar{r}^{2m}V(\bar{r},\bar{y}'')\big),-\frac{B_1}{t^{2m+1}}V(\bar{r},\bar{y}'')+\frac{B_3}{t^{N-2m+1}}\Big).
\end{equation*}
Then, we can find a solution $(t_k,\bar{r}_k,\bar{y}_k'')$ of \eqref{de1'}, \eqref{de2'} and \eqref{dela'} as before.

{\bf Case 3.} If $\bar{h}\rightarrow 0$ and $(\lambda^{\frac{N-4m}{N-2m}}\bar{h})^{-1}\rightarrow A\in (C_1,M_2)$ for some positive constant $C_1$ as $\lambda\rightarrow\infty$,
from
\eqref{dela'}, we have
\begin{equation*}
  -\frac{B_1}{t^{2m+1}}V(\bar{r},\bar{y}'')+\frac{B_3}{t^{N-2m+1}}+\frac{B_4A^{N-2m-1}}{t^{2m+1+\frac{N-4m}{N-2m}}}=o(1),\quad t\in [L_0',L_1'].
\end{equation*}
Since $N-2m+1$ and $2m+1+\frac{N-4m}{N-2m}$ are strictly greater than $2m+1$, there exists a solution of \eqref{de1'}, \eqref{de2'} and \eqref{dela'} as before.
\qed

\vspace{.5cm}

\begin{appendices}

\section{{\bf Some basic estimates}}\label{AppA}
In this section, we give some basic estimates.
\begin{lemma}\cite[Lemma B.1]{WY1}\label{AppA1}
For $y,x_i,x_j\in \mathbb{R}^N$, $i\neq j$, let
\begin{equation*}
  g_{ij}(y)=\frac{1}{(1+|y-x_i|)^{\kappa_1}}\frac{1}{(1+|y-x_j|)^{\kappa_2}},
\end{equation*}
where $\kappa_1,\kappa_2\geq 1$ are constants. Then for any constant $0<\sigma\leq \min\{\kappa_1,\kappa_2\}$, there exists a constant $C>0$ such that
\begin{equation*}
  g_{ij}(y)\leq \frac{C}{|x_i-x_j|^\sigma}\Big(\frac{1}{(1+|y-x_i|)^{\kappa_1+\kappa_2-\sigma}}+\frac{1}{(1+|y-x_j|)^{\kappa_1+\kappa_2-\sigma}}\Big).
\end{equation*}
\end{lemma}

\begin{lemma}\cite[Lemma 2.2]{GL}\label{AppA2}
For any constant $0<\sigma<N-2m$, there exists a constant $C>0$ such that
\begin{equation*}
  \int_{\mathbb{R}^N}\frac{1}{|y-z|^{N-2m}}\frac{1}{(1+|z|)^{2m+\sigma}}dz\leq \frac{C}{(1+|y|)^\sigma}.
\end{equation*}
\end{lemma}

\begin{lemma}\label{AppA3}
Assume that $N> 4m+1$, $0<\tau<2$, then there exists a small constant $\sigma>0$ such that
\begin{equation*}
  \int_{\mathbb{R}^N}\frac{1}{|y-z|^{N-2m}}Z_{\bar{r},\bar{h},\bar{y}'',\lambda}^{m^*-2}(z)\sum\limits_{j=1}^k\frac{1}{(1+\lambda|z-x_j^+|)^{\frac{N-2m}{2}+\tau}}dz\leq C\sum\limits_{j=1}^k\frac{1}{(1+\lambda|y-x_j^+|)^{\frac{N-2m}{2}+\tau+\sigma}},
\end{equation*}
and
\begin{equation*}
  \int_{\mathbb{R}^N}\frac{1}{|y-z|^{N-2m}}Z_{\bar{r},\bar{h},\bar{y}'',\lambda}^{m^*-2}(z)\sum\limits_{j=1}^k\frac{1}{(1+\lambda|z-x_j^-|)^{\frac{N-2m}{2}+\tau}}dz\leq C\sum\limits_{j=1}^k\frac{1}{(1+\lambda|y-x_j^-|)^{\frac{N-2m}{2}+\tau+\sigma}}.
\end{equation*}
\end{lemma}
\begin{proof}
The proof is similar to \cite[Lemma 2.3]{GL}, so we omit it here.
\end{proof}

\begin{lemma}\label{AppA4}
As $\lambda\rightarrow\infty$, we have
\begin{equation*}
  \frac{\partial U_{x_j^\pm,\lambda}}{\partial \lambda}=O(\lambda^{-1}U_{x_j^\pm,\lambda})+O(\lambda U_{x_j^\pm,\lambda})\frac{\partial \sqrt{1-\bar{h}^2}}{\partial \lambda}+
  O(\lambda U_{x_j^\pm,\lambda})\frac{\partial \bar{h}}{\partial \lambda}.
\end{equation*}
Hence, if $\sqrt{1-\bar{h}^2}=C\lambda^{-\beta_1}$ with $0<\beta_1<1$, we have
\begin{equation*}
  \Big|\frac{\partial U_{x_j^\pm,\lambda}}{\partial \lambda}\Big| \leq \frac{CU_{x_j^\pm,\lambda}}{\lambda^{\beta_1}}.
\end{equation*}
If $\bar{h}=C\lambda^{-\beta_2}$ with $0<\beta_2<1$, then we have
\begin{equation*}
  \Big|\frac{\partial U_{x_j^\pm,\lambda}}{\partial \lambda}\Big| \leq \frac{CU_{x_j^\pm,\lambda}}{\lambda^{\beta_2}}.
\end{equation*}
\end{lemma}
\begin{proof}
The proof is standard, we omit it.
\end{proof}

\begin{lemma}\cite[Lemma A.2]{DHWW}\label{AppA5}
For any $\gamma>0$, there exists a constant $C>0$ such that
\begin{equation*}
  \sum\limits_{j=2}^k\frac{1}{|x_j^+-x_1^+|^\gamma}\leq \frac{Ck^\gamma}{(\bar{r}\sqrt{1-\bar{h}^2})^\gamma}\sum\limits_{j=2}^k\frac{1}{(j-1)^\gamma}\leq
  \left\{
  \begin{array}{ll}
  \frac{Ck^\gamma}{(\bar{r}\sqrt{1-\bar{h}^2})^\gamma},\quad \gamma>1;\\
  \frac{Ck^\gamma \log k}{(\bar{r}\sqrt{1-\bar{h}^2})^\gamma},\quad \gamma=1;\\
  \frac{Ck}{(\bar{r}\sqrt{1-\bar{h}^2})^\gamma},\quad \gamma<1,
    \end{array}
    \right.
\end{equation*}
and
\begin{equation*}
   \sum\limits_{j=1}^k\frac{1}{|x_j^--x_1^+|^\gamma}\leq  \sum\limits_{j=2}^k\frac{1}{|x_j^+-x_1^+|^\gamma}+\frac{C}{(\bar{r}\bar{h})^\gamma}.
\end{equation*}
\end{lemma}

\begin{lemma}\label{AppA6}
Assume that $N>4m+1$, as $k\rightarrow\infty$, we have
\begin{equation}\label{A.1}\tag{A.1}
  \sum\limits_{j=2}^k\frac{1}{|x_j^+-x_1^+|^{N-2m}}=\frac{A_1k^{N-2m}}{(\sqrt{1-\bar{h}^2})^{N-2m}}\Big(1+o\big(\frac{1}{k}\big)\Big),
\end{equation}
and if $\frac{1}{\bar{h}k}=o(1)$, then
\begin{equation}\label{A.2}\tag{A.2}
   \sum\limits_{j=1}^k\frac{1}{|x_j^--x_1^+|^{N-2m}}=\frac{A_2k}{\bar{h}^{N-2m-1}(\sqrt{1-\bar{h}^2})}\Big(1+o\big(\frac{1}{\bar{h}k}\big)\Big)+O\Big(\frac{1}{(\sqrt{1-\bar{h}^2})^{N-2m}}\Big),
\end{equation}
or else, $\frac{1}{\bar{h}k}=C$, then
\begin{equation}\label{A.3}\tag{A.3}
   \sum\limits_{j=1}^k\frac{1}{|x_j^--x_1^+|^{N-2m}}=\Big(\frac{A_3k}{\bar{h}^{N-2m-1}},\frac{A_4k}{\bar{h}^{N-2m-1}}\Big),
\end{equation}
where $A_1$, $A_2$, $A_3$ and $A_4$ are some positive constants.
\end{lemma}
\begin{proof}
Without loss of generality, we assume that $k$ is even. It's easy to verify that
\begin{align*}
  \sum\limits_{j=2}^k\frac{1}{|x_j^+-x_1^+|^{N-2m}}=&\sum\limits_{j=2}^k\frac{1}{\Big(2\bar{r}\sqrt{1-\bar{h}^2}\sin \frac{(j-1)\pi}{k}\Big)^{N-2m}}\\
  =&\sum\limits_{j=2}^{\frac{k}{2}}\frac{1}{\Big(2\bar{r}\sqrt{1-\bar{h}^2}\sin \frac{(j-1)\pi}{k}\Big)^{N-2m}}+\sum\limits_{j={\frac{k}{2}}+1}^k\frac{1}{\Big(2\bar{r}\sqrt{1-\bar{h}^2}\sin \frac{(j-1)\pi}{k}\Big)^{N-2m}}.
\end{align*}
For any fixed $0<\varepsilon<\frac{1}{10}$, by a direct computation, we have
\begin{align*}
  &\sum\limits_{j=2}^{\frac{k}{2}}\frac{1}{\Big(2\bar{r}\sqrt{1-\bar{h}^2}\sin \frac{(j-1)\pi}{k}\Big)^{N-2m}}\\
  =&\sum\limits_{j=2}^{[k\varepsilon]}\frac{1}{\Big(2\bar{r}\sqrt{1-\bar{h}^2}\sin \frac{(j-1)\pi}{k}\Big)^{N-2m}}+\sum\limits^{j=\frac{k}{2}}_{[k \varepsilon]+1}\frac{1}{\Big(2\bar{r}\sqrt{1-\bar{h}^2}\sin \frac{(j-1)\pi}{k}\Big)^{N-2m}}\\
  =&\sum\limits_{j=2}^{[k \varepsilon]}\frac{1}{\Big(2\bar{r}\sqrt{1-\bar{h}^2} \frac{(j-1)\pi}{k}\Big)^{N-2m}}\Big(1+O\Big(\frac{i^2}{k^2}\Big)\Big)+
  O\Big(\int_{\varepsilon}^{\frac{1}{2}}\frac{1}{(\sqrt{1-\bar{h}^2}x)^{N-2m}}dx\Big)\\
  =&\Big(\frac{k}{\sqrt{1-\bar{h}^2}}\Big)^{N-2m}\Big(D_1+o\big(\frac{1}{k}\big)\Big),
\end{align*}
where $D_1=\frac{1}{(2\pi \bar{r})^{N-2m}}\sum\limits_{i=1}^\infty\frac{1}{i^{N-2m}}$. Using symmetry of the function $\sin x$, we can prove that
\begin{equation*}
  \sum\limits_{j={\frac{k}{2}}+1}^k\frac{1}{\Big(2\bar{r}\sqrt{1-\bar{h}^2}\sin \frac{(j-1)\pi}{k}\Big)^{N-2m}}=\Big(\frac{k}{\sqrt{1-\bar{h}^2}}\Big)^{N-2m}\Big(D_1+o\big(\frac{1}{k}\big)\Big).
\end{equation*}
This proves \eqref{A.1}. 

Similarly, we can obtain
\begin{align*}
  \sum\limits_{j=1}^k\frac{1}{|x_j^--x_1^+|^{N-2m}}=&\sum\limits_{j=1}^k\frac{1}{\Big(2\bar{r}\big[(1-\bar{h}^2)\sin^2\frac{(j-1)\pi}{k}+\bar{h}^2\big]^{\frac{1}{2}}\Big)^{N-2m}}\\
 =&\frac{2}{(2\bar{r}\bar{h})^{N-2m}}\sum\limits_{j=1}^{[k \varepsilon]}\frac{1}{\Big(\frac{1-\bar{h}^2}{\bar{h}^2}\frac{(j-1)^2\pi^2}{k^2}+1\Big)^{\frac{N-2m}{2}}}
 +O\Big(\int_{\varepsilon}^{1}\frac{1}{(\sqrt{1-\bar{h}^2}x)^{N-2m}}dx\Big).
\end{align*}
If $\frac{1}{\bar{h}k}=o(1)$, then
\begin{align*}
  \sum\limits_{j=1}^{[k \varepsilon]}\frac{1}{\Big(\frac{1-\bar{h}^2}{\bar{h}^2}\frac{(j-1)^2\pi^2}{k^2}+1\Big)^{\frac{N-2m}{2}}}=&
  \int_0^{[k \varepsilon]}\frac{1}{\Big(\frac{1-\bar{h}^2}{\bar{h}^2}\frac{\pi^2x^2}{k^2}+1\Big)^{\frac{N-2m}{2}}}dx+o(1)\\
  =&\frac{\bar{h}k}{\pi\sqrt{1-\bar{h}^2}}\int_0^{\frac{\pi\sqrt{1-\bar{h}^2}[k \varepsilon]}{\bar{h}k}}\frac{1}{(x^2+1)^{\frac{N-2m}{2}}}dx+o(1)\\
  =&\frac{\bar{h}k}{\pi\sqrt{1-\bar{h}^2}}\int_0^{+\infty}\frac{1}{(x^2+1)^{\frac{N-2m}{2}}}dx\Big(1+o\big(\frac{1}{\bar{h}k}\big)\Big),
\end{align*}
thus \eqref{A.2} holds.
As for the case $\frac{1}{\bar{h}k}=C$ for $k$ large, by the same arguments as the proof of \eqref{A.2}, we can prove that \eqref{A.3} holds. This ends the proof.
\end{proof}

For integers $i,j,m,N$ satisfying $m\geq 1$, $N>2m$, $j=1,2,\cdots,m$, denote
\begin{align*}
D(i,j)= \left\{
  \begin{array}{ll}
  0,\quad &\text{if $i<0$},\\
  1,\quad &\text{if $i=0$},\\
  \prod\limits_{h=j-i+1}^{j}(m-h),\quad &\text{if $i=1,2,\cdots, j$};\\
    \end{array}
    \right.
  \end{align*}
  \begin{align*}
E(i,j)= \left\{
  \begin{array}{ll}
  \prod\limits_{h=i}^{j-1}(N+2h),\quad &\text{if $i=0,1,\cdots,j-1$},\\
  1,\quad &\text{if $i=j$},\\
  0,\quad &\text{if $i\geq j+1$};\\
    \end{array}
    \right.
  \end{align*}
\begin{equation*}
  K_j=\prod_{h=0}^{j-1}(N-2m+2h);
\end{equation*}
\begin{equation*}
  G(i,j)=2^i\binom j i K_jD(i,j)E(i,j).
\end{equation*}
\begin{lemma}\cite[Lemma 2]{S}\label{AppA7}
For $l=1,2,\cdots,m$, there holds
\begin{equation*}
  (-\Delta)^lU_{x,\lambda}(y)=\frac{\lambda^{\frac{N-2m}{2}+2l}}{(1+\lambda^2|y-x|^2)^{\frac{N-2m}{2}+2l}}\sum\limits_{i=0}^lG(i,l)(\lambda|y-x|)^{2i}.
\end{equation*}
As a consequence,
\begin{equation*}
 \Big| \frac{\partial ^k\big((-\Delta)^lU_{x,\lambda}(y)\big)}{\partial y_{i_k}\cdots \partial y_{i_1}}\Big|\leq \frac{\lambda^{\frac{N-2m}{2}+2l+k}}{(1+\lambda^2|y-x|^2)^{\frac{N-2m}{2}+2l+k}}\sum\limits_{i=0}^lG(i,l)(\lambda|y-x|)^{2i+k}.
\end{equation*}
\end{lemma}

Suppose that $u$ and $v$ are two smooth functions in a given bounded domain $\Omega$,
define the following two bilinear functionals
\begin{equation*}
  L_{1}(u,v)=\int_\Omega \Big((-\Delta)^m u \langle y,\nabla v\rangle+(-\Delta)^m v \langle y,\nabla u\rangle\Big) dy,
\end{equation*}
\begin{equation*}
  L_{2,i}(u,v)=\int_\Omega \Big((-\Delta)^m u \frac{\partial v}{\partial y_i}+(-\Delta)^m v \frac{\partial u}{\partial y_i}\Big)dy,\quad i=1,2,\cdots,N.
\end{equation*}
We have
\begin{lemma}\cite[Proposition 3.3]{GPY}\label{AppA8}
For any integer $m>0$, there exists a function $f_{m}(u,v)$ such that
\begin{equation*}
  L_{1}(u,v)=\int_{\partial \Omega}f_{m}(u,v)dy-\frac{N-2m}{2}\int_\Omega \Big(v (-\Delta)^m u+u (-\Delta)^m v\Big)dy.
\end{equation*}
Moreover, $f_{m}(u,v)$ has the form
\begin{equation*}
  f_{m}(u,v)=\sum\limits_{j=1}^{2m-1}\bar{l}_{j}(y,\nabla ^j u,\nabla^{2m-j}v)+\sum\limits_{j=0}^{2m-1}\tilde{l}_{j}(\nabla ^j u,\nabla^{2m-j-1}v),
\end{equation*}
where $\bar{l}_{j}(y,\nabla ^j u,\nabla^{2m-j}v)$ and $\tilde{l}_{j}(\nabla ^j u,\nabla^{2m-j-1}v)$ are linear in each component.
\end{lemma}
\begin{lemma}\cite[Proposition 3.1]{GPY}\label{AppA9}
For any integer $m>0$, there exists a function $g_{m,i}(u,v)$ such that
\begin{equation*}
  L_{2,i}(u,v)=\int_{\partial \Omega}g_{m,i}(u,v)dy.
\end{equation*}
Moreover, $g_{m,i}(u,v)$ has the form
\begin{equation*}
  g_{m,i}(u,v)=\sum\limits_{j=1}^{2m-1}l_{j,i}(\nabla ^j u,\nabla^{2m-j}v),
\end{equation*}
where $l_{j,i}(\nabla ^j u,\nabla^{2m-j}v)$ is bilinear in $\nabla ^j u$ and $\nabla ^{2m-j} u$.
\end{lemma}

\end{appendices}

\begin{appendices}
\section{{\bf Energy expansion}}\label{AppB}

\begin{lemma}\label{ener2}
If $N>4m+1$, then
\begin{align*}
  \frac{\partial I(Z_{\bar{r},\bar{h},\bar{y}'',\lambda})}{\partial \lambda}=&2k\bigg(-\frac{B_1}{\lambda^{2m+1}}V(\bar{r},\bar{y}'')+\sum\limits_{j=2}^k\frac{B_2}{\lambda^{N-2m+1}|x_j^+-x_1^+|^{N-2m}}\\
  &+
  \sum\limits_{j=1}^k\frac{B_2}{\lambda^{N-2m+1}|x_j^--x_1^+|^{N-2m}}+O\Big(\frac{1}{\lambda^{2m+1+\varepsilon}}\Big)\bigg),
\end{align*}
where $B_1$ and $B_2$ are two positive constants.
\end{lemma}
\begin{proof}
By a direct computation, we have
\begin{align*}
  \frac{\partial I(Z_{\bar{r},\bar{h},\bar{y}'',\lambda})}{\partial \lambda}=&\frac{\partial I(Z^*_{\bar{r},\bar{h},\bar{y}'',\lambda})}{\partial \lambda}+O\Big(\frac{k}{\lambda^{2m+1+\varepsilon}}\Big)\\
  =&\int_{\mathbb{R}^N}V(y)Z^*_{\bar{r},\bar{h},\bar{y}'',\lambda} \frac{\partial Z^*_{\bar{r},\bar{h},\bar{y}'',\lambda}}{\partial \lambda}dy -\int_{\mathbb{R}^N}\Big((Z^*_{\bar{r},\bar{h},\bar{y}'',\lambda})^{m^*-1}-\sum\limits_{j=1}^kU_{x_j^+,\lambda}^{m^*-1}-\sum\limits_{j=1}^kU_{x_j^-,\lambda}^{m^*-1}\Big)\\
  &\times \frac{\partial Z^*_{\bar{r},\bar{h},\bar{y}'',\lambda}}{\partial \lambda}dy +O\Big(\frac{k}{\lambda^{2m+1+\varepsilon}}\Big)\\
  :=&I_1-I_2+O\Big(\frac{k}{\lambda^{2m+1+\varepsilon}}\Big).
\end{align*}
For $I_1$, by symmetry and Lemma \ref{AppA1}, we have
\begin{align*}
  I_1=&V(\bar{r},\bar{y}'')\int_{\mathbb{R}^N}Z^*_{\bar{r},\bar{h},\bar{y}'',\lambda} \frac{\partial Z^*_{\bar{r},\bar{h},\bar{y}'',\lambda}}{\partial \lambda}dy
  +\int_{\mathbb{R}^N}\big(V(y)-V(\bar{r},\bar{y}'')\big)Z^*_{\bar{r},\bar{h},\bar{y}'',\lambda} \frac{\partial Z^*_{\bar{r},\bar{h},\bar{y}'',\lambda}}{\partial \lambda}dy
  \\=&2k\bigg(V(\bar{r},\bar{y}'')\int_{\mathbb{R}^N}U_{x_1^+,\lambda}\frac{\partial U_{x_1^+,\lambda}}{\partial \lambda}dy+O\Big(\frac{1}{\lambda^{\beta_1}}\int_{\mathbb{R}^N}U_{x_1^+,\lambda}\big(\sum\limits_{j=2}^kU_{x_j^+,\lambda}+\sum\limits_{j=1}^kU_{x_j^-,\lambda}\big)dy\Big)
  +O\Big(\frac{1}{\lambda^{2m+1+\varepsilon}}\Big)
  \bigg)\\
  =&2k\bigg(\frac{V(\bar{r},\bar{y}'')}{2} \frac{\partial }{\partial  \lambda}\int_{\mathbb{R}^N}U^2_{x_1^+,\lambda}dy+O\Big(\frac{1}{\lambda^{2m+\beta_1}}\Big(\sum\limits_{j=2}^k\frac{1}{(\lambda|x_j^+-x_1^+|)^{N-4m-\varepsilon}}\Big)\\
  &+\frac{1}{\lambda^{2m+\beta_1}}\Big(\sum\limits_{j=1}^k\frac{1}{(\lambda|x_j^--x_1^+|)^{N-4m-\varepsilon}}\Big)\Big)+O\Big(\frac{1}{\lambda^{2m+1+\varepsilon}}\Big)\bigg)\\
  =&2k\Big(-\frac{B_1V(\bar{r},\bar{y}'')}{\lambda^{2m+1}}+O\Big(\frac{1}{\lambda^{2m+\beta_1+\frac{2m}{N-2m}(N-4m-\varepsilon)}}\Big)+O\Big(\frac{1}{\lambda^{2m+1+\varepsilon}}\Big)\Big)\\
  =&2k\Big(-\frac{B_1V(\bar{r},\bar{y}'')}{\lambda^{2m+1}}+O\Big(\frac{1}{\lambda^{2m+1+\varepsilon}}\Big)\Big),
\end{align*}
for some constant $B_1>0$, since $N>4m+1$ and $\iota$ is small.

Next, we estimate $I_2$. by symmetry and Lemma \ref{AppA1}, we obtain
\begin{align*}
  I_2=&2k \int_{\Omega_1^+}\Big((Z^*_{\bar{r},\bar{h},\bar{y}'',\lambda})^{m^*-1}-\sum\limits_{j=1}^kU_{x_j^+,\lambda}^{m^*-1}-\sum\limits_{j=1}^kU_{x_j^-,\lambda}^{m^*-1}\Big) \frac{\partial Z^*_{\bar{r},\bar{h},\bar{y}'',\lambda}}{\partial \lambda}dy\\
  =&2k\bigg( \int_{\Omega_1^+}(m^*-1)U_{x_1^+,\lambda}^{m^*-2}\Big(\sum\limits_{j=2}^kU_{x_j^+,\lambda}+\sum\limits_{j=1}^kU_{x_j^-,\lambda}\Big) \frac{\partial U_{x_1^+,\lambda}}{\partial \lambda}dy+O\Big(\frac{1}{\lambda^{2m+1+\varepsilon}}\Big)\bigg)\\
  =&2k\Big(-\sum\limits_{j=2}^k\frac{B_2}{\lambda^{N-2m+1}|x_j^+-x_1^+|^{N-2m}}-
  \sum\limits_{j=1}^k\frac{B_2}{\lambda^{N-2m+1}|x_j^--x_1^+|^{N-2m}}+O\Big(\frac{1}{\lambda^{2m+1+\varepsilon}}\Big)\Big),
\end{align*}
for some constant $B_2>0$. Thus, we obtain the result.
\end{proof}
\end{appendices}

\medskip

\noindent{\bfseries Data availability}
 No data was used for the research described in the article.

\end{document}